\documentclass[11pt]{article}
\usepackage[T1]{fontenc}
\usepackage[utf8]{inputenc}
\usepackage{amsmath,amssymb,amsthm,mathtools}
\usepackage{microtype}
\usepackage{enumitem}
\usepackage{booktabs}
\usepackage[mathlines]{lineno}
\usepackage{xcolor}
\usepackage[colorlinks=true,allcolors=blue]{hyperref}
\allowdisplaybreaks
\usepackage[left=2.cm,right=2.cm,top=1.8cm,bottom=1.8cm]{geometry}

\newtheorem{theorem}{Theorem}[section]
\newtheorem{proposition}[theorem]{Proposition}
\newtheorem{lemma}[theorem]{Lemma}

\newtheorem{definition}[theorem]{Definition}
\newtheorem{assumption}[theorem]{Assumption}
\theoremstyle{remark}
\newtheorem{remark}[theorem]{Remark}

\newcommand{\R}{\mathbb{R}}
\newcommand{\eps}{\varepsilon}
\newcommand{\cK}{\mathcal{K}}

\newcommand{\cL}{\mathcal{L}}
\newcommand{\cLgeoLower}{\mathcal{L}_{\mathrm{geo},*}}
\newcommand{\cLgeoUpper}{\mathcal{L}_{\mathrm{geo}}^{*}}
\newcommand{\cJlower}{\mathcal{J}_{*}}
\newcommand{\cJupper}{\mathcal{J}^{*}}
\newcommand{\norm}[1]{\left\lVert #1\right\rVert}
\newcommand{\abs}[1]{\left\lvert #1\right\rvert}

\newcommand{\tr}{\operatorname{tr}}

\newcommand{\dd}{\,\mathrm{d}}

\newcommand{\dist}{\operatorname{dist}}

\title{Regularized Subjective-Surface Flow with Monotone Reaction: Viscosity Well-Posedness and Convergence}

\author{Markjoe O. Uba\\
\small School of Mathematical and Statistical Sciences, Northern Illinois University, DeKalb, Illinois, USA\\
\small \texttt{markjoeuba@gmail.com}}
\date{}

\hypersetup{
  pdftitle={Regularized Subjective-Surface Flow with Monotone Reaction: Viscosity Well-Posedness and Convergence},
  pdfauthor={Markjoe O. Uba}
}

\begin{document}
\setlength{\abovedisplayskip}{2.5pt plus 1pt minus 1pt}
\setlength{\belowdisplayskip}{2.5pt plus 1pt minus 1pt}
\setlength{\abovedisplayshortskip}{1.5pt plus .5pt minus .5pt}
\setlength{\belowdisplayshortskip}{1.5pt plus .5pt minus .5pt}
\setlength{\jot}{2pt}
\maketitle

\begin{abstract}
We study the vanishing-regularization limit of the regularized subjective-surface
model introduced for touching and dividing cell nuclei and previously analyzed for
fixed positive regularization parameters. On a smooth bounded domain
$\Omega\subset\mathbb R^d$, the model is
\[
\begin{aligned}
\partial_tu^{\eps,\nu}
={}&\nu\Delta u^{\eps,\nu}
+A_\eps(\nabla u^{\eps,\nu})
\operatorname{div}\!\left(
G(x)\frac{\nabla u^{\eps,\nu}}
{A_\eps(\nabla u^{\eps,\nu})}
\right)
-\mu R_\eta(x,u^{\eps,\nu}),\\
&A_\eps(p)=\sqrt{\eps^2+|p|^2},
\qquad R_\eta(x,r)=\Lambda(x)H_\eta(r-q),
\end{aligned}
\]
subject to homogeneous Dirichlet data and a prescribed initial profile. We formulate
the limiting weighted level-set mean-curvature
equation as a viscosity initial-boundary value problem and establish global existence,
uniqueness, preservation of $[0,1]$, and nonexpansive dependence on the initial data.
We determine the direct upper and lower limiting values of the regularized principal
operators at zero gradient and identify the additional argument needed to recover the
geometric directional values. For every $T>0$, the regularized solutions converge
uniformly to the unique viscosity solution on $[0,T]\times\overline\Omega$ as
$\eps,\nu\to0$, independently of their relative rates of decay. This determines the
geometric limit of the regularized model used in 3D and 3D+time microscopy segmentation.
\end{abstract}



\section{Introduction}\label{sec:intro}

Microscopy images of developing tissues frequently contain nuclei whose shapes change,
separate, or approach one another closely. In such cases, portions of the interface between
neighboring nuclei may be weak or visually incomplete, and the segmentation must remain
coherent as the geometry evolves across spatial locations and successive image frames. These
features motivate a geometric evolution that combines image information with prescribed data
about nearby nuclei.

The generalized subjective-surface equation is an image-weighted level-set curvature model for
surface completion and segmentation. The image-dependent edge coefficient, constructed from
the image data, becomes small near strong image edges and thereby slows the evolving level
surface near object boundaries. Additional prescribed coefficients may encode information
obtained from thresholding or neighboring-object preprocessing; see
\cite{MikulaRemesikova2009,SartiMalladiSethian2000,umknskp2020,UbaMikulaPark2023}.

A companion paper \cite{Uba2026Regularized} studies the same regularized model for
fixed $\eps,\nu>0$ at the classical level, including global classical well-posedness,
finite-time Schauder estimates, preservation of the interval $[0,1]$, and nonexpansive
dependence on the initial datum. In the present analysis, the fixed-parameter Dirichlet
viscosity problems form the approximation family. The principal results establish global
viscosity well-posedness for the limiting geometric equation, identify the zero-gradient
selection mechanism, and prove uniform convergence of the approximation family as
$(\eps,\nu)\to(0,0)$.

For a single 3D microscopy image, $d=3$ and $x=(x_1,x_2,x_3)$. For 3D+time
data, $d=4$ and
\[
x=(x_1,x_2,x_3,\theta),
\]
where $\theta$ denotes physical time in the image data and $t$ denotes artificial segmentation
time. In both settings, the unknown remains $u(t,x)$, and the image intensity and prescribed
coefficients are defined on the corresponding data domain.

Formally, the limiting equation is
\[
 \partial_tu
 =G(x)|\nabla u|\operatorname{div}\!\left(\frac{\nabla u}{|\nabla u|}\right)
  +\nabla G(x)\cdot\nabla u-\mu R_\eta(x,u).
\]
The weighted curvature operator is undefined at points where $\nabla u=0$. In the viscosity
formulation, its value at such zero-gradient points is determined by the upper and lower
directional limits prescribed by level-set mean-curvature flow
\cite{ChenGigaGoto1991,EvansSpruck1991,Giga2006}. However, the direct upper and lower
limiting values of the regularized operators, when the gradient and regularization parameters
vanish simultaneously, do not in general coincide with these geometric directional limits.
Consequently, the standard stability theory for discontinuous operators does not by itself
identify the intended limiting equation.

The main contributions are:
\begin{enumerate}[label=(\roman*),leftmargin=2.6em]
\item a comparison principle and a Perron construction yielding global viscosity
      well-posedness for the weighted geometric equation with monotone reaction;
\item preservation of the interval $[0,1]$ and nonexpansive stability in the
      supremum norm;
\item an explicit computation of the upper and lower limiting values of the
      regularized principal operators at zero gradient, showing that the direct
      discontinuous-operator limits do not select the intended geometric operator;
\item an equivalent viscosity formulation at points where the spatial gradient vanishes and a
      stability argument that treats the geometric and reaction terms separately,
      yielding uniform convergence of the regularized solutions on
      $[0,T]\times\overline{\Omega}$ for every $T>0$, independently of the path
      along which the regularization parameters tend to zero.
\end{enumerate}

The preservation of $[0,1]$ is particularly relevant to the automatic and semi-automatic
segmentation of microscopy images, where $u$ is interpreted as a normalized segmentation
function and the object boundary is represented by the level set $\{u=0.5\}$ \cite{umknskp2020,UbaMikulaPark2023}. Nonexpansive
stability ensures that perturbations in the initial segmentation are not amplified during the
evolution.

When $\partial\Omega$ is smooth and strictly mean-convex, $G$ is constant in a neighborhood of
$\partial\Omega$, and the initial data vanish in that neighborhood, the viscosity solution
extends continuously to $[0,T]\times\overline\Omega$ and equals the homogeneous Dirichlet data
on $[0,T]\times\partial\Omega$. The limiting results hold in every dimension $d\ge2$ and for
each fixed transition width $\eta>0$ satisfying the stated threshold conditions.

Section~\ref{sec:coefficients} constructs the prescribed coefficient functions.
Section~\ref{sec:model} presents the regularized model, the limiting viscosity formulation, and
all main results. Section~\ref{sec:structural} gives the structural estimates.
Section~\ref{sec:comparison-perron} establishes comparison, boundary continuity, Perron
solvability, and stability for prescribed continuous reaction data. Section~\ref{sec:limit-analysis}
develops the vanishing-regularization analysis. Section~\ref{sec:main-proofs} proves the main
theorems, and Section~\ref{sec:conclusion} concludes the paper.

\section{Construction of the prescribed coefficient functions}\label{sec:coefficients}

The coefficient definitions below are shared with the fixed-parameter model in
\cite{Uba2026Regularized}; here they specify the prescribed data for the geometric
viscosity problem and the vanishing-regularization analysis.

\subsection{Image-dependent edge coefficient}

Let $I^0$ be the given image intensity and let $I^{TH}$ denote the locally thresholded image
constructed during preprocessing. A generalized edge detector used in subjective-surface
segmentation has the form
\begin{equation}\label{eq:edge-coefficient}
G^0(x)=g\!\left(\delta_0\left|\nabla(G_\sigma*I^0)(x)\right|
+\vartheta_0\left|\nabla(G_\sigma*I^{TH})(x)\right|\right),
\end{equation}
where $G_\sigma$ is a smoothing kernel, $g$ is positive and nonincreasing, and
$\delta_0,\vartheta_0\in[0,1]$ determine the relative influence of the original and thresholded
image information. The analysis uses the processed field, denoted by $G$, as a fixed,
smooth, and strictly positive coefficient.

\subsection{Prescribed interaction weight and monotone reaction}

Let $c_0$ denote the current candidate center and let $\mathcal N(c_0)$ be the finite set of
neighboring candidates selected during preprocessing. For each $j\in\mathcal N(c_0)$, let
$\omega_j$ be a nonnegative smooth weight and let $u_j^b$ be a prescribed function describing
the region associated with candidate $j$. A representative interaction weight is
\begin{equation}\label{eq:interaction-weight}
\Lambda(x)=\sum_{j\in\mathcal N(c_0)}
\omega_j(x)H_\eta(u_j^b(x)-q).
\end{equation}
The functions $\omega_j$ and $u_j^b$ are determined before the PDE is evolved, so $\Lambda$
remains a prescribed coefficient throughout the analysis.

Choose $q\in(0,1)$ and
\[
0<\eta<\min\{q,1-q\}.
\]
Let $H\in C^\infty(\mathbb R)$ satisfy
\[
0\le H\le1,\qquad H'\ge0,\qquad
H(s)=0\ \text{for }s\le-1,\qquad H(s)=1\ \text{for }s\ge1,
\]
and define
\begin{equation}\label{eq:profiles}
H_\eta(s)=H(s/\eta),\qquad
h_\eta(s)=H_\eta'(s)=\eta^{-1}H'(s/\eta),\qquad
P_\eta(s)=\int_{-\infty}^sH_\eta(r)\,\dd r.
\end{equation}
The reaction density and its associated interaction functional are
\begin{equation}\label{eq:reaction}
R_\eta(x,r)=\Lambda(x)H_\eta(r-q),
\end{equation}
\begin{equation}\label{eq:reaction-energy}
E_\eta(v)=\int_\Omega\Lambda(x)P_\eta(v(x)-q)\,\dd x,
\qquad
DE_\eta(v)[\phi]=\int_\Omega R_\eta(x,v)\phi\,\dd x.
\end{equation}
The monotonicity $\partial_rR_\eta=\Lambda h_\eta\ge0$ is the structural property used in
comparison and stability.

\section{The regularized subjective-surface model and geometric limit}\label{sec:model}

\subsection{The regularized model}

For $p\in\mathbb R^d$ and $\eps>0$, define
\begin{equation}\label{eq:Aeps}
A_\eps(p)=\sqrt{\eps^2+|p|^2}.
\end{equation}
For a smooth function $v$, set
\begin{equation}\label{eq:Kepsnu}
K_{\eps,\nu}[v]
=\nu\Delta v+A_\eps(\nabla v)
\operatorname{div}\!\left(G(x)\frac{\nabla v}{A_\eps(\nabla v)}\right).
\end{equation}
The regularized initial-boundary value problem is
\begin{equation}\label{eq:regularized-flow}
\boxed{\begin{cases}
\partial_tu^{\eps,\nu}=K_{\eps,\nu}[u^{\eps,\nu}]
-\mu R_\eta(x,u^{\eps,\nu}),
&(t,x)\in(0,\infty)\times\Omega,\\[1mm]
u^{\eps,\nu}=0,&(t,x)\in(0,\infty)\times\partial\Omega,\\
u^{\eps,\nu}(0,x)=u^0(x),&x\in\overline\Omega.
\end{cases}}
\end{equation}

\begin{assumption}[Data for the vanishing-regularization problem]\label{ass:data}
Let $\alpha\in(0,1)$. Assume:
\begin{enumerate}[label=(A\arabic*),leftmargin=3em]
\item $\Omega\subset\mathbb R^d$, $d\ge2$, is bounded and has $C^{4+\alpha}$ boundary,
      and $\partial\Omega$ is strictly mean-convex;
\item $G$ extends to a $C^{3+\alpha}$ function on an open neighborhood of
      $\overline\Omega$, satisfies $0<g_*\le G\le g^*$, and is constant in a neighborhood
      of $\partial\Omega$;
\item $\Lambda$ extends to a nonnegative $C^2$ function on an open neighborhood of
      $\overline\Omega$;
\item $u^0\in C^{2+\alpha}(\overline\Omega)$, $0\le u^0\le1$, and $u^0$ vanishes in a
      neighborhood of $\partial\Omega$;
\item $\mu\ge0$, while $q$, $\eta$, $H$, $H_\eta$, and $h_\eta$ satisfy the conditions in
      Section~\ref{sec:coefficients}.
\end{enumerate}
\end{assumption}

Set
\[
d_\partial(x):=\operatorname{dist}(x,\partial\Omega)
\qquad (x\in\overline\Omega).
\]
By \textup{(A2)} and \textup{(A4)}, there exists \(\delta_b>0\) such that
\[
G \text{ is constant on }
\{x\in\overline\Omega:d_\partial(x)\le2\delta_b\},
\qquad
u^0(x)=0
\quad\text{whenever }d_\partial(x)\le2\delta_b.
\]
Let \(H_{\partial\Omega}\) denote the mean curvature of \(\partial\Omega\)
computed with respect to the outward unit normal, and define
\[
H_0:=\min_{y\in\partial\Omega}H_{\partial\Omega}(y)>0.
\]
The positivity follows from strict mean convexity and compactness of
\(\partial\Omega\).

Since $u^0$ vanishes in a neighborhood of $\partial\Omega$, one has
$u^0=\nabla u^0=D^2u^0=0$ on $\partial\Omega$. Consequently,
$K_{\eps,\nu}[u^0]=0$ on $\partial\Omega$ for every $\eps,\nu>0$.
Moreover,
$R_\eta(x,u^0)=\Lambda(x)H_\eta(-q)=0
\ \text{on }\partial\Omega$.
Thus,
\[
K_{\eps,\nu}[u^0]-\mu R_\eta(x,u^0)=0
\text{ on }\partial\Omega,
\]
so the same initial datum satisfies the compatibility condition for every
choice of $\eps,\nu>0$ and may be used for any sequence
$(\eps,\nu)\to(0,0)$.

For the viscosity limit argument, we rewrite the spatial part of the
regularized equation in nondivergence form. Let $\mathbb S^d$ denote the
space of symmetric $d\times d$ matrices. For
$x\in\overline{\Omega}$, $p\in\mathbb R^d$, and $X\in\mathbb S^d$, define
\begin{equation}\label{eq:regularized-operator}
\mathcal L^{\eps,\nu}(x,p,X)
=
-\nu\operatorname{tr}X
-G(x)\left(
\operatorname{tr}X-
\frac{Xp\cdot p}{\eps^2+|p|^2}
\right)
-\nabla G(x)\cdot p.
\end{equation}
For every smooth function $v$, the equation
\[
\partial_t v
+\mathcal L^{\eps,\nu}(x,\nabla v,D^2v)
+\mu R_\eta(x,v)=0
\]
is equivalent to \eqref{eq:regularized-flow}. Consequently, every
classical solution of the regularized problem is also a viscosity
solution of
\[
u_t
+\mathcal L^{\eps,\nu}(x,\nabla u,D^2u)
+\mu R_\eta(x,u)=0.
\]
For each $\eps,\nu>0$, the operator
$\mathcal L^{\eps,\nu}$ is continuous on
$\overline{\Omega}\times\mathbb R^d\times\mathbb S^d$. This formulation
is used to identify the equation satisfied by locally uniform limits as
$(\eps,\nu)\to(0,0)$, when the limiting function need not possess
classical second derivatives.

\subsection{Geometric limit and viscosity formulation}\label{sec:viscosity}

For $p\ne0$ and $X\in\mathbb S^d$, define
\begin{equation}\label{eq:K}
\mathcal K(x,p,X)=G(x)\left(\operatorname{tr}X-\frac{Xp\cdot p}{|p|^2}\right)
+\nabla G(x)\cdot p,
\qquad \mathcal L=-\mathcal K.
\end{equation}
The limiting initial-boundary value problem is
\begin{equation}\label{eq:limit-flow}
\boxed{\begin{cases}
\partial_tu-\mathcal K(x,\nabla u,D^2u)+\mu R_\eta(x,u)=0,
&(t,x)\in(0,\infty)\times\Omega,\\
u=0,&(t,x)\in(0,\infty)\times\partial\Omega,\\
u(0,x)=u^0(x),&x\in\overline\Omega.
\end{cases}}
\end{equation}
At zero gradient, define the geometric lower and upper semicontinuous extensions,
respectively, by
\begin{align}
\cLgeoLower(x,0,X)
&:=\liminf_{\substack{(y,p,Y)\to(x,0,X)\\ p\ne0}}
\mathcal L(y,p,Y)
=-G(x)\bigl(\operatorname{tr}X-\lambda_{\min}(X)\bigr),
\label{eq:Lsubstar}\\
\cLgeoUpper(x,0,X)
&:=\limsup_{\substack{(y,p,Y)\to(x,0,X)\\ p\ne0}}
\mathcal L(y,p,Y)
=-G(x)\bigl(\operatorname{tr}X-\lambda_{\max}(X)\bigr).
\label{eq:Lstar}
\end{align}
Since $G(x)>0$ and $\lambda_{\min}(X)\le\lambda_{\max}(X)$,
\[
\cLgeoLower(x,0,X)\le\cLgeoUpper(x,0,X).
\]
For $p\ne0$, both extensions equal $\mathcal L(x,p,X)$.

For $T>0$, set
\begin{equation}\label{eq:parabolic-boundary}
Q_T=(0,T)\times\Omega,
\qquad
\partial_PQ_T=(\{0\}\times\overline\Omega)\cup([0,T)\times\partial\Omega).
\end{equation}

\begin{definition}[Dirichlet viscosity solutions]\label{def:viscosity}
Fix $T>0$, $f\in C([0,T]\times\overline\Omega)$, and
$v^0\in C(\overline\Omega)$ with $v^0=0$ on $\partial\Omega$. Define
\[
\gamma(t,x)=
\begin{cases}
v^0(x), & t=0,\quad x\in\overline\Omega,\\
0, & 0<t<T,\quad x\in\partial\Omega.
\end{cases}
\]

A bounded upper-semicontinuous function
$v:[0,T]\times\overline\Omega\to\mathbb R$ is a viscosity subsolution of
\[
\partial_t v-\mathcal K(x,Dv,D^2v)=f
\qquad\text{in }Q_T
\]
if, whenever $(t_0,x_0)\in Q_T$ and $\varphi\in C^{1,2}$ is such that
$v-\varphi$ has a local maximum at $(t_0,x_0)$, one has
\[
\partial_t\varphi(t_0,x_0)
+\cLgeoLower\!\left(
x_0,D\varphi(t_0,x_0),D^2\varphi(t_0,x_0)
\right)
-f(t_0,x_0)
\le 0.
\]

If $(t_0,x_0)\in\partial_PQ_T$ and $\varphi\in C^{1,2}$ is such that
$v-\varphi$ has a local maximum at $(t_0,x_0)$ relative to
$[0,T)\times\overline\Omega$, then
\[
\min\!\left\{
v(t_0,x_0)-\gamma(t_0,x_0),\,
\partial_t\varphi(t_0,x_0)
+\cLgeoLower\!\left(
x_0,D\varphi(t_0,x_0),D^2\varphi(t_0,x_0)
\right)
-f(t_0,x_0)
\right\}
\le 0.
\]

A bounded lower-semicontinuous function
$v:[0,T]\times\overline\Omega\to\mathbb R$ is a viscosity supersolution if,
whenever $(t_0,x_0)\in Q_T$ and $\varphi\in C^{1,2}$ is such that
$v-\varphi$ has a local minimum at $(t_0,x_0)$, one has
\[
\partial_t\varphi(t_0,x_0)
+\cLgeoUpper\!\left(
x_0,D\varphi(t_0,x_0),D^2\varphi(t_0,x_0)
\right)
-f(t_0,x_0)
\ge 0.
\]

If $(t_0,x_0)\in\partial_PQ_T$ and $\varphi\in C^{1,2}$ is such that
$v-\varphi$ has a local minimum at $(t_0,x_0)$ relative to
$[0,T)\times\overline\Omega$, then
\[
\max\!\left\{
v(t_0,x_0)-\gamma(t_0,x_0),\,
\partial_t\varphi(t_0,x_0)
+\cLgeoUpper\!\left(
x_0,D\varphi(t_0,x_0),D^2\varphi(t_0,x_0)
\right)
-f(t_0,x_0)
\right\}
\ge 0.
\]

A continuous function that is both a viscosity subsolution and a viscosity
supersolution is a Dirichlet viscosity solution. It attains the boundary data in the classical sense if
\[
v=\gamma
\qquad\text{on }\partial_PQ_T.
\]
\end{definition}

A continuous function $u$ solves \eqref{eq:limit-flow} if it is a
Dirichlet viscosity solution of
\[
\partial_tu-\mathcal K(x,Du,D^2u)=f^u
\qquad\text{in }Q_T,
\]
where the continuous function determined by $u$ is defined by
\begin{equation}\label{eq:solution-inhomogeneity}
f^u(t,x)=-\mu R_\eta\bigl(x,u(t,x)\bigr).
\end{equation}

\begin{remark}[Extension beyond the terminal time]\label{rem:closed-time-interval}
Given $v\in C([0,T]\times\overline\Omega)$, set $T^+=T+1$ and define
\[
(\mathcal E_Tv)(t,x)=\begin{cases}
v(t,x),&0\le t\le T,\\
v(T,x),&T<t\le T^+.
\end{cases}
\]
Problems used on a closed time interval are first solved on $[0,T^+)\times\overline\Omega$ and
then restricted to $[0,T]\times\overline\Omega$. Comparison shows that the restriction is
independent of the chosen continuous continuation.
\end{remark}

For $\lambda>0$, define
\[
\|z\|_{\lambda,T}=\sup_{(t,x)\in[0,T]\times\overline\Omega}e^{-\lambda t}|z(t,x)|
\]
and let
\[
X_T=\{v\in C([0,T]\times\overline\Omega):v(0,\cdot)=u^0,\ 
v=0\text{ on }[0,T]\times\partial\Omega\}.
\]
The weighted norm is equivalent to the uniform norm on each fixed time interval, so $X_T$ is
complete.

\begin{definition}[Viscosity formulation under vanishing spatial derivatives]
\label{def:vanishing-derivative-formulation}
Let $O$ be an open set with $\overline O\subset Q_T$. A bounded
upper-semicontinuous function $v$ is a viscosity subsolution of
\[
\partial_t v-\cK(x,Dv,D^2v)=f
\qquad\text{in }O
\]
in the vanishing-derivative sense if, for every
$\widehat z=(\widehat t,\widehat x)\in O$ and every
$\varphi\in C^{1,2}(O)$ such that $v-\varphi$ has a local maximum at
$\widehat z$, one has
\[
\partial_t\varphi(\widehat z)
+\cL\!\left(
\widehat x,
D_x\varphi(\widehat z),
D_x^2\varphi(\widehat z)
\right)
-f(\widehat z)
\le 0
\]
whenever $D_x\varphi(\widehat z)\neq0$, while
\[
\partial_t\varphi(\widehat z)\le f(\widehat z)
\]
whenever
\[
D_x\varphi(\widehat z)=0
\qquad\text{and}\qquad
D_x^2\varphi(\widehat z)=0.
\]

A bounded lower-semicontinuous function $v$ is a viscosity supersolution
in the vanishing-derivative sense if, for every $\widehat z\in O$ and every
$\varphi\in C^{1,2}(O)$ such that $v-\varphi$ has a local minimum at
$\widehat z$, one has
\[
\partial_t\varphi(\widehat z)
+\cL\!\left(
\widehat x,
D_x\varphi(\widehat z),
D_x^2\varphi(\widehat z)
\right)
-f(\widehat z)
\ge 0
\]
whenever $D_x\varphi(\widehat z)\neq0$, while
\[
\partial_t\varphi(\widehat z)\ge f(\widehat z)
\]
whenever
\[
D_x\varphi(\widehat z)=0
\qquad\text{and}\qquad
D_x^2\varphi(\widehat z)=0.
\]

The preceding conditions define the vanishing-derivative formulation.
Proposition~\ref{prop:vanishing-derivative-equivalence} proves its equivalence with
the full geometric viscosity formulation and, in particular, yields the geometric
viscosity inequality when
\[
D_x\varphi(\widehat z)=0
\qquad\text{and}\qquad
D_x^2\varphi(\widehat z)\neq0.
\]
\end{definition}

\subsection{Main results}\label{sec:main-results}

\begin{theorem}[Geometric well-posedness and stability]
\label{thm:wellposed}
Under Assumption~\ref{ass:data}, problem \eqref{eq:limit-flow} admits a
unique global continuous viscosity solution. For every $T>0$, the solution
belongs to $C([0,T]\times\overline\Omega)$ and agrees with the prescribed
initial and homogeneous Dirichlet data on the parabolic boundary. Moreover,
\[
0\le u(t,x)\le1
\qquad
\text{for all }(t,x)\in[0,\infty)\times\overline\Omega.
\]

If $u$ and $v$ are two solutions corresponding to the same prescribed
coefficients and to initial data $u^0$ and $v^0$, respectively, both
satisfying Assumption~\ref{ass:data}, then
\begin{equation}\label{eq:nonexpansive}
\|u(t,\cdot)-v(t,\cdot)\|_{L^\infty(\Omega)}
\le
\|u^0-v^0\|_{L^\infty(\Omega)}
\qquad\text{for every }t\ge0.
\end{equation}
\end{theorem}

\begin{proposition}[Direct joint lower and upper limits at zero gradient]\label{prop:direct-zero-gradient-limits}
For every $x\in\overline{\Omega}$ and symmetric matrix $X$, define
\begin{align}
\cJlower(x,0,X)
&:=\liminf_{\substack{(\eps,\nu,y,p,Y)\to(0,0,x,0,X)}}
\cL^{\eps,\nu}(y,p,Y)
=-G(x)\operatorname{tr}X+G(x)\min\{0,\lambda_{\min}(X)\},\label{eq:raw-lower}\\
\cJupper(x,0,X)
&:=\limsup_{\substack{(\eps,\nu,y,p,Y)\to(0,0,x,0,X)}}
\cL^{\eps,\nu}(y,p,Y)
=-G(x)\operatorname{tr}X+G(x)\max\{0,\lambda_{\max}(X)\}.\label{eq:raw-upper}
\end{align}
Then
\begin{equation}\label{eq:limiting-value-order}
\cJlower(x,0,X)\le\cLgeoLower(x,0,X)
\le\cLgeoUpper(x,0,X)\le\cJupper(x,0,X).
\end{equation}
\end{proposition}

\begin{proposition}[Equivalence of viscosity formulations]
\label{prop:vanishing-derivative-equivalence}
Let
\[
f\in C([0,T]\times\overline\Omega).
\]
For the equation
\[
\partial_t v-\cK(x,Dv,D^2v)=f
\qquad\text{in }Q_T,
\]
the geometric viscosity formulation and the formulation given in
Definition~\ref{def:vanishing-derivative-formulation} define the same class of
viscosity subsolutions. The corresponding equivalence also holds for
viscosity supersolutions.
\end{proposition}

\begin{theorem}[Path-independent vanishing-regularization limit]\label{thm:limit}
Let $(\eps_n,\nu_n)$ be any sequence of positive parameters with
$(\eps_n,\nu_n)\to(0,0)$. For each $n$, let $u^{\eps_n,\nu_n}$ be the unique global
Dirichlet viscosity solution of \eqref{eq:regularized-flow} furnished by
Proposition~\ref{prop:regularized-full}, and let $u$ be the unique global continuous
viscosity solution of \eqref{eq:limit-flow} furnished by
Theorem~\ref{thm:wellposed}. Then, for every $T>0$,
\begin{equation}\label{eq:localuniform}
\lim_{n\to\infty}
\|u^{\eps_n,\nu_n}-u\|_{C([0,T]\times\overline\Omega)}=0.
\end{equation}
The conclusion is independent of the relative rates at which $\eps_n$ and $\nu_n$ tend to zero.
\end{theorem}

\section{Structural properties of the operators}\label{sec:structural}
\begin{lemma}[Reaction bounds and admissibility]
\label{lem:reaction-inhomogeneity}
Set
\[
M_\eta=\mu\norm{\Lambda}_{L^\infty(\Omega)},\qquad
L_\eta=\norm{\Lambda}_{L^\infty(\Omega)}\norm{h_\eta}_{L^\infty(\mathbb R)}.
\]
Then
\[
0\le R_\eta(x,r)\le\norm{\Lambda}_{L^\infty(\Omega)},
\qquad
\partial_rR_\eta(x,r)=\Lambda(x)h_\eta(r-q)\ge0,
\]
and
\[
|R_\eta(x,r)-R_\eta(x,s)|\le L_\eta|r-s|.
\]
For every $v\in X_T$, define
\begin{equation}\label{eq:reaction-inhomogeneity-map}
f^{v,T}(t,x)=-\mu R_\eta\bigl(x,(\mathcal E_Tv)(t,x)\bigr).
\end{equation}
Then
\[
f^{v,T}\in C([0,T^+]\times\overline\Omega),
\qquad
-M_\eta\le f^{v,T}\le0,
\]
and
\[
f^{v,T}=0
\qquad\text{on }[0,T^+]\times\partial\Omega.
\]
Moreover,
\begin{equation}\label{eq:inhomogeneity-lipschitz}
\norm{f^{v,T}(t)-f^{w,T}(t)}_{L^\infty(\Omega)}
\le\mu L_\eta\norm{v(t)-w(t)}_{L^\infty(\Omega)}
\qquad(0\le t\le T).
\end{equation}
\end{lemma}

\begin{proof}
The bounds and monotonicity follow from $0\le H_\eta\le1$, $H_\eta'=h_\eta\ge0$, and
$\Lambda\ge0$.  The mean-value theorem gives the Lipschitz estimate.  If
$x\in\partial\Omega$ and $v\in X_T$, then $(\mathcal E_Tv)(t,x)=0$ and
$H_\eta(-q)=0$ because $q>\eta$.  Consequently,
$f^{v,T}(t,x)=0$ on $[0,T^+]\times\partial\Omega$.
\end{proof}

\begin{lemma}[Structural properties of the weighted curvature operator]
\label{lem:structural-comparison}
Fix \(T>0\) and let
\[
f\in C([0,T]\times\overline{\Omega}).
\]
For \(p\neq0\), define
\[
\mathfrak F_f(t,x,r,p,X)
=
\cL(x,p,X)-f(t,x).
\]
Then \(\mathfrak F_f\) is continuous for \(p\neq0\), degenerate elliptic in
\(X\), and independent of \(r\). Moreover, \(\cL\) is geometric in
\((p,X)\).

At \(p=0\), the lower and upper semicontinuous extensions are given by
\begin{equation}\label{eq:zero-gradient-limits}
\mathfrak F_{f,*}(t,x,r,0,X)
=
\cLgeoLower(x,0,X)-f(t,x),
\qquad
\mathfrak F_f^{\,*}(t,x,r,0,X)
=
\cLgeoUpper(x,0,X)-f(t,x).
\end{equation}
In particular,
\begin{equation}\label{eq:origin-value}
\mathfrak F_{f,*}(t,x,r,0,0)
=
\mathfrak F_f^{\,*}(t,x,r,0,0)
=
-f(t,x).
\end{equation}

The principal part satisfies
\[
\sup_{\substack{t\in[0,T]\\x\in\overline{\Omega}}}
\left|
\mathfrak F_f(t,x,r,p,X)+f(t,x)
\right|
\longrightarrow0
\qquad\text{as }(p,X)\to(0,0),
\]
uniformly with respect to \(r\in\mathbb R\).

Finally, there exists a modulus of continuity \(\omega_{f,T}\) such that,
whenever \(p\neq0\), \(\zeta>0\), and
\[
-3\zeta I
\le
\begin{pmatrix}
X&0\\
0&-Y
\end{pmatrix}
\le
3\zeta
\begin{pmatrix}
I&-I\\
-I&I
\end{pmatrix},
\]
one has, for all \(s,t\in[0,T]\),
\begin{equation}\label{eq:ishii-structural-estimate}
\mathfrak F_f(t,x,r,p,X)
-
\mathfrak F_f(s,y,r,p,Y)
\ge
-\omega_{f,T}\!\left(
|t-s|
+
|x-y|(|p|+1)
+
\zeta|x-y|^2
\right).
\end{equation}
\end{lemma}

\begin{proof}
For \(p\neq0\), define
\[
\widehat p=\frac{p}{|p|},
\qquad
P_p=I-\widehat p\otimes\widehat p,
\qquad
\sigma=\sqrt{G}.
\]
The representation \eqref{eq:K} and the continuity of \(f\) imply that
\(\mathfrak F_f\) is continuous on the set \(\{p\neq0\}\). If
\(X\le Y\), then \(\operatorname{tr}(P_pX)\le\operatorname{tr}(P_pY)\);
therefore
\[
\mathfrak F_f(t,x,r,p,X)
\ge
\mathfrak F_f(t,x,r,p,Y),
\]
which proves degenerate ellipticity in the stated sign convention.

Fix \((t,x,r,X)\), and let
\[
(s_j,y_j,r_j,p_j,Y_j)\to(t,x,r,0,X),
\qquad p_j\neq0.
\]
For every \(j\),
\[
\mathfrak F_f(s_j,y_j,r_j,p_j,Y_j)
=
\cL(y_j,p_j,Y_j)-f(s_j,y_j).
\]
Since \(f(s_j,y_j)\to f(t,x)\), the elementary identities
\[
\liminf_j(a_j-b_j)=\liminf_j a_j-b,
\qquad
\limsup_j(a_j-b_j)=\limsup_j a_j-b
\]
whenever \(b_j\to b\), yield \eqref{eq:zero-gradient-limits}. Setting
\(X=0\) and using
\(\cLgeoLower(x,0,0)=\cLgeoUpper(x,0,0)=0\) gives \eqref{eq:origin-value}.

Let \(\lambda>0\), \(\rho\in\mathbb R\), and
\(Y=\lambda X+\rho p\otimes p\). A direct calculation gives
\[
\operatorname{tr}Y
-
\frac{Y(\lambda p)\cdot(\lambda p)}{|\lambda p|^2}
=
\lambda\left(
\operatorname{tr}X-\frac{Xp\cdot p}{|p|^2}
\right),
\]
because the two terms containing \(\rho|p|^2\) cancel. The first-order
part is homogeneous of degree one in \(p\); hence
\[
\cL\bigl(x,\lambda p,\lambda X+\rho p\otimes p\bigr)
=
\lambda\cL(x,p,X).
\]
Hence \(\cL\) satisfies the geometric homogeneity identity.

For \(p\neq0\),
\[
\left|\frac{Xp\cdot p}{|p|^2}\right|\le\|X\|,
\]
and consequently
\[
|\cL(x,p,X)|
\le
(d+1)g^*\|X\|
+
\|\nabla G\|_{L^\infty(\Omega)}|p|.
\]
Because \(\mathfrak F_f+f=\cL\), the estimate implies
\[
\sup_{\substack{t\in[0,T]\\x\in\overline\Omega}}
\left|\mathfrak F_f(t,x,r,p,X)+f(t,x)\right|
\longrightarrow0
\qquad\text{as }(p,X)\to(0,0),
\]
uniformly in \(r\). The case \(p=0\), \(X\neq0\), is governed by
\eqref{eq:zero-gradient-limits}.

It remains to establish \eqref{eq:ishii-structural-estimate}. Let
\(e_1,\ldots,e_{d-1}\) be an orthonormal basis of
\(\operatorname{range}P_p\). Applying the upper block-matrix inequality
to \((\sigma(x)e_\ell,\sigma(y)e_\ell)\) and summing over
\(\ell=1,\ldots,d-1\) gives
\[
G(x)\operatorname{tr}(P_pX)
-
G(y)\operatorname{tr}(P_pY)
\le
3(d-1)\zeta|\sigma(x)-\sigma(y)|^2
\le
C\zeta|x-y|^2.
\]
Since \(G\in C^2(\overline\Omega)\),
\[
-\bigl(\nabla G(x)-\nabla G(y)\bigr)\cdot p
\ge
-C|x-y||p|.
\]
Uniform continuity of \(f\) on the compact set
\([0,T]\times\overline\Omega\) yields a nondecreasing modulus
\(\omega_f\) such that
\[
-f(t,x)+f(s,y)
\ge
-\omega_f\bigl(|t-s|+|x-y|\bigr).
\]
Define
\[
\Theta
=
|t-s|+|x-y|(|p|+1)+\zeta|x-y|^2.
\]
These three inequalities imply
\[
\mathfrak F_f(t,x,r,p,X)
-
\mathfrak F_f(s,y,r,p,Y)
\ge
-2C\Theta-\omega_f(\Theta).
\]
Thus \eqref{eq:ishii-structural-estimate} holds with
\[
\omega_{f,T}(\rho)=2C\rho+\omega_f(\rho),
\qquad \rho\ge0.
\]

Finally, \(\mathfrak F_f\) is independent of \(r\), so for all \(r\ge q\),
\[
\mathfrak F_f(t,x,r,p,X)-\mathfrak F_f(t,x,q,p,X)=0
\ge0\,(r-q).
\]
Hence the operator is proper with monotonicity constant \(0\). Continuity
for \(p\neq0\), degenerate ellipticity, \eqref{eq:origin-value}, and
\eqref{eq:ishii-structural-estimate} verify the assumptions of
\cite[Theorem~3.6.4]{Giga2006}. The identities
\eqref{eq:zero-gradient-limits} determine the lower and upper limiting
extensions whenever \(p=0\) and \(X\neq0\).
\end{proof}

\section{Comparison, boundary behavior, and Perron construction}\label{sec:comparison-perron}
\begin{remark}[The inhomogeneous term is incorporated in the comparison operator]
\label{rem:inhomogeneity-comparison-operator}
The comparison theorem is applied to
$\mathfrak F_f=\cL-f$.  The identities \eqref{eq:zero-gradient-limits} determine the zero-gradient extensions required by \cite[Theorem~3.6.4]{Giga2006}, with the
continuous inhomogeneous term incorporated into the comparison operator.
\end{remark}

\begin{theorem}[Comparison with ordered parabolic boundary values]
\label{thm:comparison}
Fix a terminal time $T>0$ and let
$f\in C([0,T]\times\overline\Omega)$.  If $v$ is a bounded relaxed Dirichlet viscosity subsolution and $w$
a bounded relaxed Dirichlet viscosity supersolution of
\[
 \partial_t z-\cK(x,Dz,D^2z)=f
 \qquad\text{in }Q_T,
\]
in the sense of Definition~\ref{def:viscosity}, and assume explicitly that
\[
 v^*(\zeta)\le w_*(\zeta)
 \qquad\text{for every }\zeta\in\partial_P Q_T,
\]
then $v\le w$ in $Q_T$.
\end{theorem}

\begin{proof}
Define
\[
F(t,x,r,p,X)=\mathfrak F_f(t,x,r,p,X)
=
\cL(x,p,X)-f(t,x).
\]
Lemma~\ref{lem:structural-comparison} verifies the continuity condition for
\(p\neq0\), degenerate ellipticity, properness, the finite value at
\((p,X)=(0,0)\), and the matrix-continuity estimate required by
\cite[Theorem~3.6.4]{Giga2006}. The hypothesis
\[
v^*(\zeta)\le w_*(\zeta)
\qquad(\zeta\in\partial_PQ_T)
\]
is the boundary ordering assumed in that theorem. Therefore the comparison
result applies to \(v\) and \(w\) and gives \(v\le w\) in \(Q_T\).
\end{proof}

\begin{theorem}[Comparison for the geometric equation with monotone reaction]
\label{thm:reaction-comparison}
Fix \(T>0\). Let \(u\) be a bounded upper-semicontinuous viscosity
subsolution and \(v\) a bounded lower-semicontinuous viscosity supersolution of
\[
\partial_t z+\cL(x,Dz,D^2z)+\mu R_\eta(x,z)=0
\qquad\text{in }Q_T,
\]
with the geometric extensions \(\cLgeoLower\) and \(\cLgeoUpper\) at zero
gradient. If
\[
u^*(\zeta)\le v_*(\zeta)\qquad(\zeta\in\partial_PQ_T),
\]
then \(u\le v\) in \(Q_T\).
\end{theorem}

\begin{proof}
For \(p\ne0\), set
\[
\mathfrak F_\eta(x,r,p,X)
:=\cL(x,p,X)+\mu R_\eta(x,r).
\]
Its lower and upper extensions at \(p=0\) are
\[
\mathfrak F_{\eta,*}(x,r,0,X)
=\cLgeoLower(x,0,X)+\mu R_\eta(x,r),\qquad
\mathfrak F_\eta^{*}(x,r,0,X)
=\cLgeoUpper(x,0,X)+\mu R_\eta(x,r).
\]
If \(r\ge s\), then
\[
\mathfrak F_\eta(x,r,p,X)-\mathfrak F_\eta(x,s,p,X)
=\mu\Lambda(x)\bigl(H_\eta(r-q)-H_\eta(s-q)\bigr)\ge0,
\]
so the operator is proper. Furthermore, since \(0\le H_\eta\le1\),
\[
|R_\eta(x,r)-R_\eta(y,r)|
\le \|D\Lambda\|_{L^\infty(\Omega)}|x-y|.
\]
Adding this estimate to \eqref{eq:ishii-structural-estimate} gives the
matrix-continuity modulus required by
\cite[Theorem~3.6.4]{Giga2006}. The remaining hypotheses follow from
Lemma~\ref{lem:structural-comparison}; hence that comparison theorem gives
\(u\le v\) in \(Q_T\).
\end{proof}

\begin{remark}[Role of the ordered-boundary hypothesis]
\label{rem:ordered-boundary-values}
The boundary hypothesis in Theorem~\ref{thm:comparison} is the semicontinuous ordering $v^*(\zeta)\le w_*(\zeta)$ on $\partial_PQ_T$.  It is an independent comparison hypothesis for general relaxed Dirichlet viscosity sub- and supersolutions.  In the applications below, the boundary ordering follows either from equality with the prescribed Dirichlet data or from the strict boundary inequalities imposed on the comparison functions.
\end{remark}

\begin{lemma}[Fixed-parameter structural comparison]
\label{lem:fixed-parameter-structural-comparison}
Fix \(\eps,\nu>0\), \(S>0\), and
\(h\in C([0,S]\times\overline\Omega)\). Define
\[
Q_\eps(p):=\frac{p\otimes p}{\eps^2+|p|^2},\qquad
B_\eps(p):=I-Q_\eps(p),\qquad
A^{\eps,\nu}(x,p):=\nu I+G(x)B_\eps(p),
\]
and
\[
F_h^{\eps,\nu}(t,x,r,p,X)
:=-\operatorname{tr}\!\bigl(A^{\eps,\nu}(x,p)X\bigr)
-DG(x)\cdot p-h(t,x).
\]
Then \(F_h^{\eps,\nu}\) is continuous, proper, degenerate elliptic, and
satisfies condition \cite[(3.14)]{CrandallIshiiLions1992}. Consequently,
comparison holds for bounded viscosity sub- and supersolutions of the
corresponding Cauchy--Dirichlet problem on every bounded parabolic
domain.
\end{lemma}

\begin{proof}
Since \(0\le Q_\eps(p)\le I\),
\[
0\le B_\eps(p)\le I,\qquad
\nu I\le A^{\eps,\nu}(x,p)\le(\nu+g^*)I.
\]
The operator is continuous and independent of \(r\). Moreover,
\[
X\le Y
\quad\Longrightarrow\quad
F_h^{\eps,\nu}(t,x,r,p,X)
\ge F_h^{\eps,\nu}(t,x,r,p,Y),
\]
so it is degenerate elliptic and proper.

Let \(\alpha>0\), set \(p=\alpha(x-y)\), and suppose
\(X,Y\in\mathbb S^d\) satisfy
\[
-3\alpha\begin{pmatrix}I&0\\0&I\end{pmatrix}
\le
\begin{pmatrix}X&0\\0&-Y\end{pmatrix}
\le
3\alpha\begin{pmatrix}I&-I\\-I&I\end{pmatrix}.
\]
Put
\(\Sigma^{\eps,\nu}(x,p):=A^{\eps,\nu}(x,p)^{1/2}\). The preceding
matrix inequality implies
\begin{equation}\label{eq:fixed-parameter-ishii-diffusion}
\operatorname{tr}\!\bigl(A^{\eps,\nu}(x,p)X\bigr)
-\operatorname{tr}\!\bigl(A^{\eps,\nu}(y,p)Y\bigr)
\le3\alpha\left\|
\Sigma^{\eps,\nu}(x,p)-\Sigma^{\eps,\nu}(y,p)
\right\|_F^2.
\end{equation}
Because the two matrices are affine functions of the common matrix
\(B_\eps(p)\), they are simultaneously diagonalizable. For every
eigenvalue \(b\in[0,1]\) of \(B_\eps(p)\),
\[
\left|\sqrt{\nu+G(x)b}-\sqrt{\nu+G(y)b}\right|
=\frac{b|G(x)-G(y)|}
{\sqrt{\nu+G(x)b}+\sqrt{\nu+G(y)b}}
\le\frac{|G(x)-G(y)|}{2\sqrt\nu}.
\]
Hence
\begin{equation}\label{eq:fixed-parameter-square-root-modulus}
\left\|
\Sigma^{\eps,\nu}(x,p)-\Sigma^{\eps,\nu}(y,p)
\right\|_F^2
\le\frac{d\|DG\|_{L^\infty(\Omega)}^2}{4\nu}|x-y|^2.
\end{equation}
Let \(\omega_h\) be a modulus of continuity of \(h\) on
\([0,S]\times\overline\Omega\). Using
\eqref{eq:fixed-parameter-ishii-diffusion},
\eqref{eq:fixed-parameter-square-root-modulus}, and
\(|DG(x)-DG(y)|\le\|D^2G\|_{L^\infty}|x-y|\), one obtains
\begin{align*}
&F_h^{\eps,\nu}(t,y,r,p,Y)-F_h^{\eps,\nu}(t,x,r,p,X)\\
&\quad\le C_\nu\alpha|x-y|^2+\omega_h(|x-y|),\qquad
C_\nu:=\frac{3d\|DG\|_{L^\infty}^2}{4\nu}
+\|D^2G\|_{L^\infty}.
\end{align*}
Thus condition (3.14) holds with
\[
\omega_{\eps,\nu,h}(\rho):=C_\nu\rho+\omega_h(\rho),
\qquad \omega_{\eps,\nu,h}(\rho)\to0
\quad(\rho\downarrow0).
\]
The comparison conclusion follows from
\cite[Theorem~8.2]{CrandallIshiiLions1992}.
\end{proof}

\begin{lemma}[Strict comparison with ordered local boundary values]
\label{lem:strict-relaxed-boundary-comparison}
Let \(D=I\times O\subset Q_T\) be a bounded relative parabolic domain,
where \(I\) is an interval and \(O\subset\Omega\) is open. Let
\(\mathfrak L\) denote either \(\cL\), with extensions
\(\cLgeoLower,\cLgeoUpper\), or a fixed operator \(\cL^{\eps,\nu}\) with
\(\eps,\nu>0\). Write
\[
\partial_PD=\Gamma_{\mathrm{data}}\cup\Gamma_{\mathrm{art}},\qquad
\Gamma_{\mathrm{data}}:=\partial_PD\cap\partial_PQ_T.
\]
On an intersection of these two sets, all conditions below are imposed.

Suppose that \(u\) is a bounded upper-semicontinuous viscosity subsolution
of
\[
\partial_tu+\mathfrak L(x,Du,D^2u)=f\qquad\text{in }D
\]
and satisfies
\[
\limsup_{D\ni z\to\zeta}u(z)\le\gamma(\zeta)
\qquad(\zeta\in\Gamma_{\mathrm{data}}).
\]
Let \(b^+\in C^{1,2}(\overline D)\), and assume that for some
\(\delta_0>0\),
\[
\partial_tb^++\mathfrak L^{*}(x,Db^+,D^2b^+)\ge f+\delta_0
\quad\text{in }D,\qquad
b^+\ge\gamma+\delta_0
\quad\text{on }\Gamma_{\mathrm{data}},
\]
and
\[
\limsup_{D\ni z\to\zeta}u(z)\le b^+(\zeta)
\qquad(\zeta\in\Gamma_{\mathrm{art}}).
\]
Then \(u\le b^+\) in \(D\).

Dually, suppose that \(u\) is a bounded lower-semicontinuous viscosity
supersolution satisfying
\[
\gamma(\zeta)\le\liminf_{D\ni z\to\zeta}u(z)
\qquad(\zeta\in\Gamma_{\mathrm{data}}).
\]
If \(b^-\in C^{1,2}(\overline D)\) satisfies
\[
\partial_tb^-+\mathfrak L_{*}(x,Db^-,D^2b^-)\le f-\delta_0
\quad\text{in }D,\qquad
b^-\le\gamma-\delta_0
\quad\text{on }\Gamma_{\mathrm{data}},
\]
and
\[
b^-(\zeta)\le\liminf_{D\ni z\to\zeta}u(z)
\qquad(\zeta\in\Gamma_{\mathrm{art}}),
\]
then \(b^-\le u\) in \(D\). In the regularized case,
\(\mathfrak L_*=\mathfrak L^*=\mathfrak L\).
\end{lemma}

\begin{proof}
The strict differential inequality makes \(b^+\) a viscosity
supersolution of the same equation in \(D\). The assumptions on
\(\Gamma_{\mathrm{data}}\) and \(\Gamma_{\mathrm{art}}\) give
\[
\limsup_{D\ni z\to\zeta}u(z)\le b^+(\zeta)
\qquad(\zeta\in\partial_PD).
\]
For the geometric operator, the proof of
Theorem~\ref{thm:comparison} applies on \(D\), because its
structural hypotheses are local and Lemma~\ref{lem:structural-comparison}
holds on every bounded subdomain. For the fixed regularized operator,
Lemma~\ref{lem:fixed-parameter-structural-comparison} gives comparison on
\(D\). In either case, comparison
gives \(u\le b^+\) in \(D\).

The lower assertion follows by applying the same argument to \(b^-\) as a
subsolution and \(u\) as a supersolution, using the displayed lower
boundary ordering.
\end{proof}

\begin{remark}[Geometry of the local domains]
\label{rem:artificial-edges}
No regularity of \(\Gamma_{\mathrm{art}}\) is required. The comparison
argument uses only the semicontinuous ordering on the full parabolic
boundary. In the applications below, every local set is chosen as a
bounded relative parabolic domain.
\end{remark}

\begin{lemma}[Perron closure under local boundary comparison]
\label{lem:perron-boundary-closure}
Fix \(T>0\), let \(\mathfrak L\) denote either \(\cL\), with its geometric
extensions, or a fixed \(\cL^{\eps,\nu}\) with \(\eps,\nu>0\), and consider
\[
\partial_tz+\mathfrak L(x,Dz,D^2z)=f\qquad\text{in }Q_T
\]
with continuous datum \(\gamma\) on \(\partial_PQ_T\). Assume comparison
for bounded sub- and supersolutions and the standard stability properties of
viscosity subsolutions under finite maxima and upper semicontinuous
regularization of locally bounded suprema.

Let \(\underline z\) and \(\overline z\) be bounded global sub- and
supersolutions with \(\underline z\le\overline z\) in \(Q_T\) and
\[
\limsup_{Q_T\ni z\to\zeta}\underline z(z)\le\gamma(\zeta)
\le\liminf_{Q_T\ni z\to\zeta}\overline z(z)
\qquad(\zeta\in\partial_PQ_T).
\]
Suppose that for every \(z_b\in\partial_PQ_T\) and \(\rho>0\) there are a
bounded relative parabolic domain \(D_b\subset Q_T\), a number
\(\sigma_{b,\rho}>0\), and functions
\(\beta^-_{b,\rho},\beta^+_{b,\rho}\in C^{1,2}(\overline{D_b})\) such that:
\begin{enumerate}[label=\textup{(B\arabic*)},leftmargin=2.8em]
\item
\[
\partial_t\beta^-_{b,\rho}
+\mathfrak L_*(x,D\beta^-_{b,\rho},D^2\beta^-_{b,\rho})
\le f-\sigma_{b,\rho},
\quad
\partial_t\beta^+_{b,\rho}
+\mathfrak L^*(x,D\beta^+_{b,\rho},D^2\beta^+_{b,\rho})
\ge f+\sigma_{b,\rho}
\]
in \(D_b\).
\item With
\[
\Gamma_{b,\mathrm{data}}:=\partial_PD_b\cap\partial_PQ_T,
\qquad
\partial_PD_b=\Gamma_{b,\mathrm{data}}\cup\Gamma_{b,\mathrm{art}},
\]
one has
\[
\beta^-_{b,\rho}\le\gamma-\sigma_{b,\rho}
\le\gamma+\sigma_{b,\rho}\le\beta^+_{b,\rho}
\quad\text{on }\Gamma_{b,\mathrm{data}},
\]
\[
\beta^-_{b,\rho}\le\underline z,
\qquad \overline z\le\beta^+_{b,\rho}
\quad\text{on }\Gamma_{b,\mathrm{art}},
\qquad
\beta^-_{b,\rho}\le\overline z,
\quad \underline z\le\beta^+_{b,\rho}
\quad\text{in }D_b.
\]
\item
\[
\liminf_{Q_T\ni z\to z_b}\beta^-_{b,\rho}(z)
\ge\gamma(z_b)-\rho,
\qquad
\limsup_{Q_T\ni z\to z_b}\beta^+_{b,\rho}(z)
\le\gamma(z_b)+\rho.
\]
\end{enumerate}
Then the Perron function is the unique bounded continuous Dirichlet
viscosity solution and attains \(\gamma\) on \(\partial_PQ_T\).
\end{lemma}

\begin{proof}
Let \(\mathcal S\) consist of the bounded upper-semicontinuous subsolutions
\(z\) satisfying
\[
\underline z\le z\le\overline z\quad\text{in }Q_T,
\qquad
\limsup_{Q_T\ni\xi\to\zeta}z(\xi)\le\gamma(\zeta)
\quad(\zeta\in\partial_PQ_T).
\]
The boundary hypothesis on \(\underline z\) gives
\(\underline z\in\mathcal S\), so \(\mathcal S\ne\varnothing\). Define
\[
Z:=\sup_{z\in\mathcal S}z,
\qquad Z^*:=\limsup{}^*Z,
\qquad Z_*:=\liminf{}_*Z.
\]
The stability result for upper semicontinuous regularizations of locally
bounded suprema, as used in \cite[Theorem~4.1 and Remark~4.5]{CrandallIshiiLions1992}, gives that \(Z^*\) is a subsolution.
We verify that \(Z_*\) is a supersolution. Let
\(z_0=(t_0,x_0)\in Q_T\), and suppose that \(Z_*-\varphi\) has a strict
local minimum equal to zero at \(z_0\), with \(\varphi\in C^{1,2}\). If
the supersolution inequality failed, then
\(Z_*(z_0)<\overline z(z_0)\). Indeed, equality would imply that
\(\overline z-\varphi\) has a local minimum at \(z_0\), and the
supersolution inequality for \(\overline z\) would contradict the assumed
failure. Hence, after choosing a bounded parabolic domain \(C\) with
\(\overline C\subset Q_T\), there exist \(\eta,\delta>0\) such that
\[
\partial_t\varphi+\mathfrak L^*(x,D\varphi,D^2\varphi)-f\le-2\eta
\quad\text{in }C,
\qquad
\varphi+2\delta\le Z_*\quad\text{on }\partial_PC,
\qquad
\varphi+2\delta\le\overline z\quad\text{in }\overline C.
\]
By compactness of \(\partial_PC\) and the definition of \(Z_*\), finitely
many members of \(\mathcal S\) have a maximum \(z_\delta\in\mathcal S\)
that satisfies
\(z_\delta\ge\varphi+\delta\) on \(\partial_PC\). The function
\[
\widetilde z_\delta=
\begin{cases}
\max\{z_\delta,\varphi+\delta\},&\text{in }C,\\
z_\delta,&\text{in }Q_T\setminus C,
\end{cases}
\]
is a subsolution because
\(\mathfrak L_*\le\mathfrak L^*\), belongs to \(\mathcal S\), and exceeds
\(Z\) at points
approaching \(z_0\), a contradiction. Hence \(Z_*\) is a supersolution.

Fix \(z_b\in\partial_PQ_T\) and \(\rho>0\). For every
\(z\in\mathcal S\), conditions (B1)--(B2) and
Lemma~\ref{lem:strict-relaxed-boundary-comparison} give
\[
z\le\beta^+_{b,\rho}\qquad\text{in }D_b.
\]
Taking the supremum and then the upper semicontinuous regularization, (B3)
yields
\[
\limsup_{Q_T\ni z\to z_b}Z^*(z)\le\gamma(z_b)+\rho.
\]
Similarly,
\[
\widetilde\beta^-_{b,\rho}:=
\begin{cases}
\max\{\underline z,\beta^-_{b,\rho}\},&\text{in }D_b,\\
\underline z,&\text{in }Q_T\setminus D_b
\end{cases}
\]
belongs to \(\mathcal S\), and (B3) gives
\[
\liminf_{Q_T\ni z\to z_b}Z_*(z)\ge\gamma(z_b)-\rho.
\]
Applying these inequalities with \(\rho=1/n\) gives
\[
Z^*(z_b)\le\gamma(z_b)\le Z_*(z_b).
\]
Since \(Z_*\le Z^*\), equality holds on \(\partial_PQ_T\). The assumed
comparison principle applied to \((Z^*,Z_*)\) gives \(Z^*\le Z_*\) in
\(Q_T\); hence \(Z=Z^*=Z_*\) is continuous, solves the equation, and
attains the prescribed datum. Uniqueness follows from comparison.
\end{proof}

\begin{lemma}[Uniform boundary-distance estimate]
\label{lem:boundary-distance-control}
There exist constants
\[
\delta\in(0,\delta_b]
\qquad\text{and}\qquad
\kappa_0>0
\]
such that
\begin{equation}
\label{eq:distance-curvature}
\cK\bigl(x,Dd_\partial(x),D^2d_\partial(x)\bigr)
\le -\kappa_0
\end{equation}
for every \(x\in\Omega\) satisfying
\[
0<d_\partial(x)<\delta.
\]

Let \(T,M>0\), and suppose that
\[
f\in C([0,T]\times\overline{\Omega}),
\qquad
-M\le f\le0
\quad\text{in }[0,T]\times\overline{\Omega},
\]
and
\[
f(t,y)=0
\qquad
\text{for every }(t,y)\in(0,T]\times\partial\Omega.
\]
Let \(v\in C([0,T]\times\overline\Omega)\) be a Dirichlet viscosity
solution attaining the prescribed initial and spatial boundary data and
satisfying
\[
\partial_t v-\cK(x,Dv,D^2v)=f
\qquad\text{in }(0,T]\times\Omega.
\]
with
\[
v(0,x)=v^0(x),
\qquad
0\le v^0\le1
\quad\text{in }\overline{\Omega},
\]
and assume that
\[
v^0(x)=0
\qquad
\text{whenever }d_\partial(x)\le2\delta_b.
\]
Then there exists a constant \(A_{M,T}>0\), independent of \(v\), such
that
\begin{equation}
\label{eq:boundary-distance-bound}
-A_{M,T}d_\partial(x)
\le v(t,x)
\le A_{M,T}d_\partial(x)
\end{equation}
for every
\[
(t,x)\in[0,T]\times\Omega
\qquad\text{with}\qquad
0<d_\partial(x)<\delta.
\]
In particular, the estimate gives the quantitative boundary modulus
\[
\lim_{\Omega\ni x\to y}v(t,x)=0
\qquad
\text{for every }(t,y)\in[0,T]\times\partial\Omega.
\]
\end{lemma}

\begin{proof}
Since \(\partial\Omega\) is of class \(C^{4+\alpha}\), there exists
\(\delta_0>0\) such that
\[
d_\partial\in
C^{4+\alpha}\bigl(
\{x\in\Omega:0<d_\partial(x)<\delta_0\}
\bigr).
\]
Moreover,
\[
|Dd_\partial(x)|=1,
\qquad
D^2d_\partial(x)\,Dd_\partial(x)=0
\]
whenever \(0<d_\partial(x)<\delta_0\).

Let
\[
\kappa_1(y),\ldots,\kappa_{d-1}(y)
\]
denote the principal curvatures of \(\partial\Omega\) at
\(y\in\partial\Omega\), computed with respect to the outward unit normal,
and define
\[
K_\partial
:=
\max_{y\in\partial\Omega}
\max_{1\le\ell\le d-1}
|\kappa_\ell(y)|.
\]
For \(s\in(0,\delta_0)\), the principal curvatures of the level set
\[
\{x\in\Omega:d_\partial(x)=s\}
\]
at the point corresponding to \(y\in\partial\Omega\) are
\[
\kappa_\ell(s,y)
=
\frac{\kappa_\ell(y)}
     {1-s\kappa_\ell(y)},
\qquad
1\le\ell\le d-1.
\]
Hence its mean curvature is
\[
H_s(y)
=
\sum_{\ell=1}^{d-1}
\frac{\kappa_\ell(y)}
     {1-s\kappa_\ell(y)}.
\]

Choose
\[
0<\delta\le\min\{\delta_0,\delta_b\}
\]
such that
\[
\delta K_\partial\le\frac12
\qquad\text{and}\qquad
2(d-1)\delta K_\partial^2\le\frac{H_0}{2}.
\]
For \(0<s<\delta\), the first inequality gives
\[
|1-s\kappa_\ell(y)|\ge\frac12.
\]
Therefore,
\begin{align*}
|H_s(y)-H_{\partial\Omega}(y)|
&=
\left|
\sum_{\ell=1}^{d-1}
\frac{s\kappa_\ell(y)^2}
     {1-s\kappa_\ell(y)}
\right| \le
2s\sum_{\ell=1}^{d-1}|\kappa_\ell(y)|^2 \le
2(d-1)sK_\partial^2 \le
\frac{H_0}{2}.
\end{align*}
Combining
\[
H_{\partial\Omega}(y)\ge H_0
\]
with the preceding estimate gives
\[
H_s(y)\ge\frac{H_0}{2}.
\]

The outward unit normal to the level set
\(\{d_\partial=s\}\) is \(-Dd_\partial\). Consequently,
\[
\operatorname{tr}\bigl(D^2d_\partial\bigr)
=
\operatorname{tr}
\left(
\bigl(I-Dd_\partial\otimes Dd_\partial\bigr)
D^2d_\partial
\right)
=
-H_s.
\]
Because \(G\) is constant on
\[
\{x\in\Omega:0<d_\partial(x)<\delta_b\},
\]
one has \(DG=0\) there. Thus, for
\(0<d_\partial(x)<\delta\),
\begin{align*}
\cK\bigl(x,Dd_\partial,D^2d_\partial\bigr)
&=
-G(x)H_{d_\partial(x)}(x) \le
-\frac{g_*H_0}{2}.
\end{align*}
Therefore, \eqref{eq:distance-curvature} holds with
\[
\kappa_0:=\frac{g_*H_0}{2}.
\]

For \(\sigma>0\), define
\[
\overline q_\sigma(t,x)
:=
1+\sigma(1+t)
\]
and
\[
\underline q_\sigma(t,x)
:=
-Mt-\sigma(1+t).
\]
A direct calculation gives
\[
\partial_t\overline q_\sigma=\sigma,
\qquad
D\overline q_\sigma=0,
\qquad
D^2\overline q_\sigma=0,
\]
and
\[
\partial_t\underline q_\sigma=-M-\sigma,
\qquad
D\underline q_\sigma=0,
\qquad
D^2\underline q_\sigma=0.
\]
Consequently,
\[
\partial_t\overline q_\sigma
-
\cK\bigl(x,D\overline q_\sigma,D^2\overline q_\sigma\bigr)
\ge f+\sigma
\]
and
\[
\partial_t\underline q_\sigma
-
\cK\bigl(x,D\underline q_\sigma,D^2\underline q_\sigma\bigr)
\le f-\sigma.
\]
The corresponding initial inequalities and the inequalities on
\((0,T]\times\partial\Omega\) required by
Lemma~\ref{lem:strict-relaxed-boundary-comparison} are also satisfied.
Hence
\[
-Mt-\sigma(1+t)
\le v(t,x)
\le
1+\sigma(1+t)
\]
for every
\[
(t,x)\in[0,T]\times\overline{\Omega}.
\]
Applying this inequality with \(\sigma=1/n\), \(n\in\mathbb N\), gives
\[
-Mt-\frac{1+t}{n}
\le v(t,x)
\le
1+\frac{1+t}{n}
\]
for every \(n\). Since these inequalities hold for every \(n\),
\begin{equation}
\label{eq:global-v-bounds}
-Mt\le v(t,x)\le1
\qquad
\text{on }[0,T]\times\overline{\Omega}.
\end{equation}

Choose \(A_{M,T}>0\) so that
\[
A_{M,T}\delta\ge1+MT
\qquad\text{and}\qquad
A_{M,T}\kappa_0\ge M.
\]
Set
\[
D_\delta
:=
(0,T)\times
\{x\in\Omega:0<d_\partial(x)<\delta\},
\]
and, for \(\sigma>0\), define
\[
\overline b_\sigma(t,x)
:=
A_{M,T}d_\partial(x)+\sigma(1+t)
\]
and
\[
\underline b_\sigma(t,x)
:=
-A_{M,T}d_\partial(x)-\sigma(1+t).
\]
The homogeneity of \(\cK\) and
\eqref{eq:distance-curvature} imply
\[
\partial_t\overline b_\sigma
-
\cK\bigl(x,D\overline b_\sigma,D^2\overline b_\sigma\bigr)
\ge f+\sigma
\]
and
\[
\partial_t\underline b_\sigma
-
\cK\bigl(x,D\underline b_\sigma,D^2\underline b_\sigma\bigr)
\le f-\sigma
\]
in \(D_\delta\).

On the portion of \(\partial_PD_\delta\) contained in
\(\partial_PQ_T\), the initial condition and the condition on
\((0,T]\times\partial\Omega\) give
\[
\underline b_\sigma\le\gamma-\sigma
\qquad\text{and}\qquad
\gamma+\sigma\le\overline b_\sigma.
\]
On the remaining portion
\[
(0,T)\times\{x\in\Omega:d_\partial(x)=\delta\},
\]
the estimate \eqref{eq:global-v-bounds} and the choice
\(A_{M,T}\delta\ge1+MT\) yield
\[
\underline b_\sigma\le v\le\overline b_\sigma.
\]
Lemma~\ref{lem:strict-relaxed-boundary-comparison} therefore gives
\[
-A_{M,T}d_\partial(x)-\sigma(1+t)
\le v(t,x)
\le
A_{M,T}d_\partial(x)+\sigma(1+t)
\]
for every \((t,x)\in D_\delta\).

Taking \(\sigma=1/n\) gives
\[
-A_{M,T}d_\partial(x)-\frac{1+t}{n}
\le v(t,x)
\le
A_{M,T}d_\partial(x)+\frac{1+t}{n}
\]
for every \(n\in\mathbb N\). These inequalities imply
\[
-A_{M,T}d_\partial(x)
\le v(t,x)
\le
A_{M,T}d_\partial(x),
\]
which proves \eqref{eq:boundary-distance-bound}. Finally, if
\(y\in\partial\Omega\), then \(d_\partial(x)\to0\) as
\(\Omega\ni x\to y\), and hence
\[
\lim_{\Omega\ni x\to y}v(t,x)=0.
\]
\end{proof}

\begin{lemma}[Uniform initial-time estimate for the geometric equation]\label{lem:geometric-initial-control}
Let $w^0\in C^2(\overline{\Omega})$ have zero boundary value, and let
$f\in C([0,T]\times\overline{\Omega})$ satisfy $\norm{f}_{L^\infty}\le M$.  There is a
constant $C_0$, depending on $w^0$, $M$, and the fixed coefficients, such that every continuous
Dirichlet viscosity solution attaining the homogeneous data on $(0,T)\times\partial\Omega$
\[
\partial_t w-\cK(x,\nabla w,D^2w)=f,\qquad
w=0\ \text{on }(0,T)\times\partial\Omega,
\]
with initial datum $w^0$ satisfies
\[
\norm{w(t,\cdot)-w^0}_{L^\infty(\Omega)}\le C_0t
\qquad(0\le t\le T).
\]
\end{lemma}

\begin{proof}
Define
\[
B_0
=
\sup_{x\in\overline\Omega}
\max\left\{
|\cLgeoUpper(x,Dw^0,D^2w^0)|,
|\cLgeoLower(x,Dw^0,D^2w^0)|
\right\}.
\]
The formulas for the limiting extensions imply
\[
B_0
\le
(d+1)g^*\|D^2w^0\|_{L^\infty(\Omega)}
+
\|DG\|_{L^\infty(\Omega)}\|Dw^0\|_{L^\infty(\Omega)}.
\]
Set \(C_0=B_0+M+1\). For \(\sigma>0\), let
\[
\overline b_\sigma(t,x)
=w^0(x)+C_0t+\sigma(1+t),
\qquad
\underline b_\sigma(t,x)
=w^0(x)-C_0t-\sigma(1+t).
\]
At every point, including points where \(Dw^0=0\),
\[
\partial_t\overline b_\sigma
+
\cLgeoUpper(x,D\overline b_\sigma,D^2\overline b_\sigma)-f
\ge\sigma,
\]
\[
\partial_t\underline b_\sigma
+
\cLgeoLower(x,D\underline b_\sigma,D^2\underline b_\sigma)-f
\le-\sigma.
\]
Furthermore,
\[
\underline b_\sigma\le\gamma-\sigma,
\qquad
\gamma+\sigma\le\overline b_\sigma
\quad\text{on }\partial_PQ_T.
\]
Lemma~\ref{lem:strict-relaxed-boundary-comparison} gives
\[
w^0(x)-C_0t-\sigma(1+t)
\le w(t,x)
\le
w^0(x)+C_0t+\sigma(1+t).
\]
Taking \(\sigma=1/n\), \(n\in\mathbb N\), and using the resulting
inequalities for every \(n\) gives
\[
\|w(t,\cdot)-w^0\|_{L^\infty(\Omega)}\le C_0t.
\]
\end{proof}

\begin{proposition}[Perron solvability and continuous attainment of the Dirichlet data]
\label{prop:dirichlet}
Fix $T>0$ and let $f\in C([0,T]\times\overline\Omega)$ satisfy
\[
-M\le f\le0,
\qquad
f=0\quad\text{on }(0,T]\times\partial\Omega.
\]
Let $v^0\in C^2(\overline{\Omega})$ satisfy
\[
0\le v^0\le1\ \text{in }\overline\Omega,\qquad
v^0=0\ \text{on }\partial\Omega,\qquad
v^0(x)=0\ \text{whenever }d_\partial(x)\le2\delta_b.
\]
Then the parabolic Dirichlet problem in
Definition~\ref{def:viscosity} has a unique bounded continuous \emph{Dirichlet}
viscosity solution $v$ with
\[
-Mt\le v(t,x)\le1\qquad ((t,x)\in[0,T]\times\overline{\Omega}).
\]
The solution is unique in the Dirichlet class and is interpreted at the terminal time according to Remark~\ref{rem:closed-time-interval}.
\end{proposition}

\begin{proof}
Assume that \(v^0(x)=0\) whenever \(d_\partial(x)\le2\delta_b\). Set
\(T^+=T+1\) and replace \(f\) by the continuation \(\mathcal E_Tf\) from
Remark~\ref{rem:closed-time-interval}. Since \(f(T,y)=0\) for every
\(y\in\partial\Omega\),
\[
(\mathcal E_Tf)(t,y)=0\qquad ((t,y)\in(0,T^+]\times\partial\Omega).
\]
We construct the solution on \([0,T^+)\times\overline\Omega\) and subsequently
restrict it to \([0,T]\times\overline\Omega\).

Define
\[
B_0:=\sup_{x\in\overline\Omega}\max\!\left\{
\left|\cLgeoLower(x,Dv^0,D^2v^0)\right|,
\left|\cLgeoUpper(x,Dv^0,D^2v^0)\right|\right\},
\qquad C_0:=B_0+M+1.
\]
The formulas for \(\cLgeoLower\) and \(\cLgeoUpper\) imply \(B_0<\infty\). In the
fixed-point argument, \(M=M_\eta:=\mu\|\Lambda\|_{L^\infty(\Omega)}<\infty\) is fixed before the
present proposition is applied. Set
\[
\underline z(t,x):=v^0(x)-C_0t,\qquad
\overline z(t,x):=v^0(x)+C_0t.
\]
For every \((t,x)\in Q_{T^+}\), including points at which \(Dv^0(x)=0\),
\[
\partial_t\underline z+\cLgeoLower(x,D\underline z,D^2\underline z)-f
\le-C_0+B_0+M=-1,
\qquad
\partial_t\overline z+\cLgeoUpper(x,D\overline z,D^2\overline z)-f
\ge C_0-B_0\ge1.
\]
Moreover, \(\underline z\le\gamma\le\overline z\) on
\(\partial_PQ_{T^+}\). Hence \(\underline z\) and \(\overline z\) are the
global sub- and supersolutions used in Lemma~\ref{lem:perron-boundary-closure}.

We verify conditions \textup{(B1)}--\textup{(B3)} of that lemma.

\noindent\emph{Case 1: \(z_b=(0,x_b)\) with \(x_b\in\Omega\).}
Choose \(r_b>0\) such that
\[
0<r_b<T^+,
\qquad \overline{B_{r_b}(x_b)}\subset\Omega,
\qquad D_{b,r_b}:=(0,r_b)\times B_{r_b}(x_b).
\]
Write
\[
\Gamma_{b,\mathrm{data}}:=\{0\}\times\overline{B_{r_b}(x_b)},
\qquad
\Gamma_{b,\mathrm{art}}:=\partial_PD_{b,r_b}\setminus
\Gamma_{b,\mathrm{data}}.
\]
Fix \(B_{b,\rho}>0\) and set
\[
\psi^-_{b,\rho}(x):=v^0(x)-B_{b,\rho}|x-x_b|^2,
\qquad
\psi^+_{b,\rho}(x):=v^0(x)+B_{b,\rho}|x-x_b|^2.
\]
Define
\begin{align*}
M_{b,\rho}
&:=\sup_{x\in\overline{B_{r_b}(x_b)}}
\max\!\left\{
\left|\cLgeoLower(x,D\psi^-_{b,\rho},D^2\psi^-_{b,\rho})\right|,
\left|\cLgeoUpper(x,D\psi^+_{b,\rho},D^2\psi^+_{b,\rho})\right|\right\},\\
C_{b,\rho}&:=\max\{C_0,M+M_{b,\rho}+1\}.
\end{align*}
Then \(M_{b,\rho}<\infty\). For
\[
\beta^-_{b,\rho}(t,x):=\psi^-_{b,\rho}(x)-C_{b,\rho}t,
\qquad
\beta^+_{b,\rho}(t,x):=\psi^+_{b,\rho}(x)+C_{b,\rho}t,
\]
one has in \(D_{b,r_b}\)
\[
\partial_t\beta^-_{b,\rho}-\cK(x,D\beta^-_{b,\rho},D^2\beta^-_{b,\rho})\le f-1,
\qquad
\partial_t\beta^+_{b,\rho}-\cK(x,D\beta^+_{b,\rho},D^2\beta^+_{b,\rho})\ge f+1.
\]
Furthermore,
\[
\beta^-_{b,\rho}-\underline z
=-B_{b,\rho}|x-x_b|^2-(C_{b,\rho}-C_0)t\le0,
\qquad
\beta^+_{b,\rho}-\overline z
=B_{b,\rho}|x-x_b|^2+(C_{b,\rho}-C_0)t\ge0
\]
in \(D_{b,r_b}\), and
\[
\beta^-_{b,\rho}(0,x)\le v^0(x)\le\beta^+_{b,\rho}(0,x)
\qquad (x\in\overline{B_{r_b}(x_b)}).
\]

\noindent\emph{Case 2: \(z_b=(t_b,y_b)\) with
\(t_b\in(0,T^+)\) and \(y_b\in\partial\Omega\).}
Let \(\delta\in(0,\delta_b]\) and \(\kappa_0>0\) be furnished by
Lemma~\ref{lem:boundary-distance-control}. Choose
\[
0<r_b<\min\!\left\{\delta,\delta_b,\frac{t_b}{2},T^+-t_b\right\}
\]
so that, with
\[
D_{b,r_b}:=\{(t,x)\in Q_{T^+}:|t-t_b|<r_b,\ |x-y_b|<r_b\},
\]
one has
\[
D_{b,r_b}\subset(t_b-r_b,t_b+r_b)\times
\{x\in\Omega:0<d_\partial(x)<\delta\},
\qquad
\overline{D_{b,r_b}}\cap(\{0\}\times\overline\Omega)=\varnothing.
\]
Set
\[
\Gamma_{b,\mathrm{data}}:=\partial_PD_{b,r_b}\cap
((0,T^+)\times\partial\Omega),
\qquad
\Gamma_{b,\mathrm{art}}:=\partial_PD_{b,r_b}\setminus
\Gamma_{b,\mathrm{data}}.
\]
The datum equals zero on \(\Gamma_{b,\mathrm{data}}\), and
\(d_\partial(x)\le|x-y_b|\) for every \(x\in\Omega\). Choose
\begin{equation}\label{eq:boundary-B-choice}
B_{b,\rho}>0,\qquad B_{b,\rho}r_b^2\ge C_0T^++1.
\end{equation}
For \(A>0\), define
\[
\beta^\pm(t,x):=\pm A d_\partial(x)
\pm B_{b,\rho}\bigl(|x-y_b|^2+(t-t_b)^2\bigr),
\qquad P(p):=I-\widehat p\otimes\widehat p.
\]
If \(|p\mp ADd_\partial|\le A/2\), then
\[
\|P(p)-P(\pm ADd_\partial)\|\le\frac4A|p\mp ADd_\partial|.
\]
Because \(DG=0\) on \(\{0<d_\partial<\delta_b\}\) and
\(D^2d_\partial\) is bounded there, constants \(C_1,C_2>0\), independent
of \(A,B_{b,\rho},r_b\), satisfy
\begin{equation}
\label{eq:boundary-perturbation}
\left|\cK(x,D\beta^\pm,D^2\beta^\pm)
-\cK(x,\pm ADd_\partial,\pm AD^2d_\partial)\right|
\le C_1B_{b,\rho}+C_2B_{b,\rho}r_b
\end{equation}
whenever \(2B_{b,\rho}r_b\le A/2\). Choose \(A_{b,\rho}>0\) such that
\begin{equation}
\label{eq:boundary-A-choice}
A_{b,\rho}\ge4B_{b,\rho}r_b,
\qquad
A_{b,\rho}\kappa_0\ge
M+1+C_1B_{b,\rho}+(C_2+2)B_{b,\rho}r_b.
\end{equation}
Set
\[
\beta^\pm_{b,\rho}(t,x):=\pm A_{b,\rho}d_\partial(x)
\pm B_{b,\rho}\bigl(|x-y_b|^2+(t-t_b)^2\bigr).
\]
The homogeneity and oddness of \(\cK\), together with
\eqref{eq:distance-curvature}, give
\[
\cK(x,A_{b,\rho}Dd_\partial,A_{b,\rho}D^2d_\partial)
\le-A_{b,\rho}\kappa_0,
\qquad
\cK(x,-A_{b,\rho}Dd_\partial,-A_{b,\rho}D^2d_\partial)
\ge A_{b,\rho}\kappa_0.
\]
Consequently, \eqref{eq:boundary-perturbation}--\eqref{eq:boundary-A-choice}
and \(|\partial_t\beta^\pm_{b,\rho}|\le2B_{b,\rho}r_b\) imply
\[
\partial_t\beta^+_{b,\rho}-\cK(x,D\beta^+_{b,\rho},D^2\beta^+_{b,\rho})
\ge1\ge f+1,
\qquad
\partial_t\beta^-_{b,\rho}-\cK(x,D\beta^-_{b,\rho},D^2\beta^-_{b,\rho})
\le-M-1\le f-1.
\]
Moreover,
\[
|D\beta^\pm_{b,\rho}|\ge A_{b,\rho}-2B_{b,\rho}r_b
\ge\frac{A_{b,\rho}}2>0,
\qquad
\beta^-_{b,\rho}\le0\le\beta^+_{b,\rho}
\quad\text{on }\Gamma_{b,\mathrm{data}}.
\]
On \(\Gamma_{b,\mathrm{art}}\), the definition of \(B_{b,\rho}\) gives
\[
\beta^-_{b,\rho}\le-C_0T^+\le\underline z,
\qquad
\overline z\le1+C_0T^+\le\beta^+_{b,\rho}.
\]
Since \(d_\partial(x)<\delta\le\delta_b\) in \(D_{b,r_b}\), the assumption
on \(v^0\) gives \(v^0=0\) there, and hence
\[
\beta^-_{b,\rho}\le0\le\overline z,
\qquad
\underline z\le0\le\beta^+_{b,\rho}
\quad\text{in }D_{b,r_b}.
\]

\noindent\emph{Case 3: \(z_b=(0,y_b)\) with \(y_b\in\partial\Omega\).}
Choose \(r_b\in(0,\min\{\delta_b,T^+\})\) and set
\[
D_{b,r_b}:=(0,r_b)\times(\Omega\cap B_{r_b}(y_b)),
\]
\[
\Gamma_{b,\mathrm{data}}:=\partial_PD_{b,r_b}\cap
\left[(\{0\}\times\overline\Omega)\cup((0,T^+)\times\partial\Omega)\right],
\qquad
\Gamma_{b,\mathrm{art}}:=\partial_PD_{b,r_b}\setminus
\Gamma_{b,\mathrm{data}}.
\]
Fix \(B_{b,\rho}>0\), define
\[
\psi^-_{b,\rho}(x):=v^0(x)-B_{b,\rho}|x-y_b|^2,
\qquad
\psi^+_{b,\rho}(x):=v^0(x)+B_{b,\rho}|x-y_b|^2,
\]
and let
\[
M_{b,\rho}:=\sup_{x\in\overline\Omega\cap\overline{B_{r_b}(y_b)}}
\max\!\left\{
\left|\cLgeoLower(x,D\psi^-_{b,\rho},D^2\psi^-_{b,\rho})\right|,
\left|\cLgeoUpper(x,D\psi^+_{b,\rho},D^2\psi^+_{b,\rho})\right|\right\}.
\]
Set \(C_{b,\rho}:=\max\{C_0,M+M_{b,\rho}+1\}\).
For
\[
\beta^-_{b,\rho}(t,x):=v^0(x)-B_{b,\rho}|x-y_b|^2-C_{b,\rho}t,
\qquad
\beta^+_{b,\rho}(t,x):=v^0(x)+B_{b,\rho}|x-y_b|^2+C_{b,\rho}t,
\]
one has
\[
\partial_t\beta^-_{b,\rho}-\cK(x,D\beta^-_{b,\rho},D^2\beta^-_{b,\rho})\le f-1,
\qquad
\partial_t\beta^+_{b,\rho}-\cK(x,D\beta^+_{b,\rho},D^2\beta^+_{b,\rho})\ge f+1.
\]
On the two components of \(\Gamma_{b,\mathrm{data}}\), respectively,
\[
\beta^-_{b,\rho}(0,x)\le v^0(x)\le\beta^+_{b,\rho}(0,x),
\qquad
\beta^-_{b,\rho}(t,x)\le0\le\beta^+_{b,\rho}(t,x).
\]
Equivalently, \(\beta^-_{b,\rho}\le\gamma\le\beta^+_{b,\rho}\) on
\(\Gamma_{b,\mathrm{data}}\). In addition,
\[
\beta^-_{b,\rho}-\underline z
=-B_{b,\rho}|x-y_b|^2-(C_{b,\rho}-C_0)t\le0,
\qquad
\beta^+_{b,\rho}-\overline z
=B_{b,\rho}|x-y_b|^2+(C_{b,\rho}-C_0)t\ge0
\]
in \(D_{b,r_b}\), and therefore on \(\Gamma_{b,\mathrm{art}}\).

Fix \(\rho>0\). For each \(z_b\in\partial_PQ_{T^+}\), choose the functions
constructed in the corresponding case with the parameter \(\rho/2\), select
\[
0<\sigma_{b,\rho}<\min\{\rho/2,1\},
\qquad
\widetilde\beta^-_{b,\rho}:=\beta^-_{b,\rho/2}-\sigma_{b,\rho},
\qquad
\widetilde\beta^+_{b,\rho}:=\beta^+_{b,\rho/2}+\sigma_{b,\rho}.
\]
Since \(\cK\) is independent of the function value, the differential
inequalities remain unchanged. The shifts also preserve the inequalities on
\(\Gamma_{b,\mathrm{art}}\) and give
\[
\widetilde\beta^-_{b,\rho}\le\gamma-\sigma_{b,\rho},
\qquad
\gamma+\sigma_{b,\rho}\le\widetilde\beta^+_{b,\rho}
\quad\text{on }\Gamma_{b,\mathrm{data}}.
\]
For \(z_b=(0,y_b)\) with \(y_b\in\partial\Omega\), these inequalities are,
more explicitly,
\[
\widetilde\beta^-_{b,\rho}(0,x)
\le v^0(x)-\sigma_{b,\rho},
\qquad
v^0(x)+\sigma_{b,\rho}
\le\widetilde\beta^+_{b,\rho}(0,x),
\]
on \(\partial_PD_{b,r_b}\cap(\{0\}\times\overline\Omega)\), and
\[
\widetilde\beta^-_{b,\rho}(t,x)\le-\sigma_{b,\rho},
\qquad
\sigma_{b,\rho}\le\widetilde\beta^+_{b,\rho}(t,x)
\]
on \(\partial_PD_{b,r_b}\cap((0,T^+)\times\partial\Omega)\). Finally,
continuity and the identities \(\beta^-_{b,\rho/2}(z_b)
=\beta^+_{b,\rho/2}(z_b)=\gamma(z_b)\) give
\[
\lim_{Q_{T^+}\ni z\to z_b}\widetilde\beta^-_{b,\rho}(z)
=\gamma(z_b)-\sigma_{b,\rho}>\gamma(z_b)-\rho,
\qquad
\lim_{Q_{T^+}\ni z\to z_b}\widetilde\beta^+_{b,\rho}(z)
=\gamma(z_b)+\sigma_{b,\rho}<\gamma(z_b)+\rho.
\]
Thus \textup{(B1)}--\textup{(B3)} hold.

Lemma~\ref{lem:perron-boundary-closure} yields a bounded continuous Dirichlet
solution \(\widehat v\) on \([0,T^+)\times\overline\Omega\). Theorem~\ref{thm:comparison}
implies uniqueness on \(Q_S\) for every \(S<T^+\). Define
\[
v:=\widehat v\big|_{[0,T]\times\overline\Omega}.
\]
Then \(v\in C([0,T]\times\overline\Omega)\), and Theorem~\ref{thm:comparison}
on \(Q_T\) shows that \(v\) is independent of the chosen continuation of
\(f\). Comparison with \(-Mt\) and \(1\) gives
\[
-Mt\le v(t,x)\le1\qquad ((t,x)\in[0,T]\times\overline\Omega).
\]
\end{proof}

\begin{lemma}[Stability under prescribed inhomogeneities]\label{lem:inhomogeneous-stability}
Let $f,g\in C([0,T]\times\overline{\Omega})$ vanish on
$(0,T]\times\partial\Omega$ and satisfy $-M\le f,g\le0$.  Let $v_f$ and $v_g$ be the associated Dirichlet viscosity solutions
with initial data $v_f^0$ and $v_g^0$, where both initial data belong to the same
class of initial data satisfying $v^0=0$ on $\{d_\partial\le2\delta_b\}$ used in Proposition~\ref{prop:dirichlet}.  Then
\begin{equation}\label{eq:inhomogeneous-stability}
\norm{v_f(t)-v_g(t)}_{L^\infty}
\le \norm{v_f^0-v_g^0}_{L^\infty}
+\int_0^t\norm{f(s)-g(s)}_{L^\infty}\,\dd s.
\end{equation}
\end{lemma}

\begin{proof}
Define
\[
a(t)=\|v_f^0-v_g^0\|_{L^\infty(\Omega)}
+
\int_0^t\|f(s)-g(s)\|_{L^\infty(\Omega)}\,\mathrm ds.
\]
The map
\(t\mapsto\|f(t)-g(t)\|_{L^\infty(\Omega)}\) is continuous; hence
\(a\in C^1([0,T])\) and
\[
a'(t)=\|f(t)-g(t)\|_{L^\infty(\Omega)},
\qquad -a'(t)\le f(t,x)-g(t,x)\le a'(t).
\]
Because \(a\) is independent of \(x\),
\[
D_x(v_g\pm a)=D_xv_g,
\qquad D_x^2(v_g\pm a)=D_x^2v_g.
\]
If \(v_g+a-\varphi\) has a local minimum at \((t_0,x_0)\), then
\(v_g-(\varphi-a)\) has a local minimum there. The supersolution inequality
for \(v_g\) gives
\[
\partial_t\varphi(t_0,x_0)
+\cLgeoUpper(x_0,D_x\varphi,D_x^2\varphi)-f(t_0,x_0)
\ge a'(t_0)+g(t_0,x_0)-f(t_0,x_0)\ge0.
\]
Thus \(v_g+a\) is a supersolution of the equation with inhomogeneity
\(f\). The corresponding local-maximum calculation gives
\[
\partial_t\varphi(t_0,x_0)
+\cLgeoLower(x_0,D_x\varphi,D_x^2\varphi)-f(t_0,x_0)
\le g(t_0,x_0)-f(t_0,x_0)-a'(t_0)\le0,
\]
so \(v_g-a\) is a subsolution.

On $(0,T]\times\partial\Omega$,
\[
(v_g-a)^*=-a\le0=(v_f)_*,
\qquad
(v_f)^*=0\le a=(v_g+a)_*.
\]
At \(t=0\), the definition of \(a(0)\) gives
\[
v_g^0-a(0)\le v_f^0\le v_g^0+a(0).
\]
Theorem~\ref{thm:comparison} applied to the pairs
\((v_g-a,v_f)\) and \((v_f,v_g+a)\) yields
\[
v_g-a\le v_f\le v_g+a
\qquad\text{in }Q_T.
\]
Taking the \(L^\infty(\Omega)\)-norm at time \(t\) gives
\eqref{eq:inhomogeneous-stability}.
\end{proof}

\section{Vanishing-regularization analysis}\label{sec:limit-analysis}
Recall the regularized nondivergence-form operator defined in \eqref{eq:regularized-operator}.
For every smooth field $\psi$, the divergence-form principal part and
\eqref{eq:regularized-operator} agree identically:
\begin{align}
&-\nu\Delta\psi
-A_\eps(D\psi)\operatorname{div}\!\left(
G\frac{D\psi}{A_\eps(D\psi)}\right)\notag\\
&\qquad =
-\nu\operatorname{tr}(D^2\psi)
-G\left(\operatorname{tr}(D^2\psi)
-\frac{D^2\psi\,D\psi\cdot D\psi}{\eps^2+|D\psi|^2}\right)
-\nabla G\cdot D\psi\notag\\
&\qquad =\cL^{\eps,\nu}(x,D\psi,D^2\psi).
\label{eq:divergence-nondivergence}
\end{align}
Accordingly, throughout this paper a regularized solution means a viscosity
solution of the nondivergence-form equation
$\partial_t w+\cL^{\eps,\nu}(x,Dw,D^2w)=f$.
For every $\psi\in C^2(\overline\Omega)$, identity~\eqref{eq:divergence-nondivergence} shows that the nondivergence-form operator evaluated at $(D\psi,D^2\psi)$ equals the displayed divergence expression.  The analysis proceeds in the
nondivergence-form viscosity framework.  At a nonzero gradient,
\eqref{eq:regularized-operator} converges locally uniformly to $\cL$.  At zero gradient the
direct upper and lower limiting values are different from the geometric upper and lower limiting extensions in
\eqref{eq:Lsubstar}--\eqref{eq:Lstar}.

\begin{lemma}[Uniform operator convergence away from zero gradient]
\label{lem:away-zero-gradient-convergence}
For every $\rho,B>0$ and $\eps,\nu>0$,
\[
 \sup_{\substack{x\in\overline\Omega,\ |p|\ge\rho,\\ \|X\|\le B}}
 \left|\cL^{\eps,\nu}(x,p,X)-\cL(x,p,X)\right|
 \le B\left(d\nu+g^*\frac{\eps^2}{\rho^2}\right).
\]
Consequently, the convergence is uniform on every compact subset of
$\{(x,p,X):p\ne0\}$.
\end{lemma}

\begin{proof}
The drift term is identical in the two operators.  For $p\ne0$, their difference is bounded by
\[
 \nu|\tr X|+G(x)\|X\|\,
 \left|\frac{|p|^2}{\eps^2+|p|^2}-1\right|
 \le d\nu\|X\|+g^*\|X\|\frac{\eps^2}{|p|^2},
\]
which gives the assertion.
\end{proof}

\begin{lemma}[Explicit fixed-regularization Perron functions]
\label{lem:regularized-local-perron-functions}
Fix $\eps,\nu>0$, a terminal time $S>0$, and data $w^0,f$ in the class of
Proposition~\ref{prop:regularized-dirichlet}.  Put
\[
 B_0^{\eps,\nu}:=
 \norm{\cL^{\eps,\nu}(x,Dw^0,D^2w^0)}_{L^\infty(\Omega)},
 \qquad C_0^{\eps,\nu}:=B_0^{\eps,\nu}+M+2.
\]
Then
\[
 \underline z=w^0-C_0^{\eps,\nu}t,
 \qquad
 \overline z=w^0+C_0^{\eps,\nu}t
\]
are global strict sub- and supersolutions.  Moreover, for every parabolic-boundary point
$z_b$ and every $\rho>0$, there are a relative neighborhood $D_b$ and functions
$\beta^-_{b,\rho},\beta^+_{b,\rho}\in C^{1,2}(\overline{D_b})$ satisfying
\textup{(B1)}--\textup{(B3)} of Lemma~\ref{lem:perron-boundary-closure} for the fixed
regularized equation, with $\underline z,\overline z$ as global lower and upper functions.
\end{lemma}

\begin{proof}
The matrix quotient in \eqref{eq:regularized-operator} is continuous for fixed $\eps>0$;
hence $B_0^{\eps,\nu}<\infty$.  Since $-M\le f\le0$, the definition of $C_0^{\eps,\nu}$ gives
\[
 \partial_t\underline z+\cL^{\eps,\nu}(x,D\underline z,D^2\underline z)
 \le f-1,
 \qquad
 \partial_t\overline z+\cL^{\eps,\nu}(x,D\overline z,D^2\overline z)
 \ge f+1.
\]
They satisfy
\[
 \underline z(0,x)=w^0(x)=\overline z(0,x)\quad (x\in\overline\Omega),
 \qquad
 \underline z(t,y)=-C_0^{\eps,\nu}t\le0\le C_0^{\eps,\nu}t=\overline z(t,y)
 \quad ((t,y)\in(0,S]\times\partial\Omega).
\]
Since \(w^0(y)=0\) for every \(y\in\partial\Omega\), the two identities agree on
\(\{0\}\times\partial\Omega\).

At an initial point $z_b=(0,x_b)$, choose the bounded relative parabolic domain
$D_b=(0,r)\times(\Omega\cap B_r(x_b))$ and, after fixing $B>0$,
set
\[
 \psi^-_B=w^0-B|x-x_b|^2,
 \qquad
 \psi^+_B=w^0+B|x-x_b|^2.
\]
Set
\[
M_B^{\eps,\nu}
=
\max_{\pm}
\left\|\cL^{\eps,\nu}(x,D\psi_B^\pm,D^2\psi_B^\pm)\right\|_{L^\infty}.
\]
Choose
\[
C\ge\max\{C_0^{\eps,\nu},M_B^{\eps,\nu}+M+1\}.
\]
Then the functions
\[
 \beta^-_B(t,x)=\psi^-_B(x)-Ct,
 \qquad
 \beta^+_B(t,x)=\psi^+_B(x)+Ct
\]
satisfy
\[
\partial_t\beta^-_B
+\cL^{\eps,\nu}(x,D\beta^-_B,D^2\beta^-_B)
\le f-1,
\]
\[
\partial_t\beta^+_B
+\cL^{\eps,\nu}(x,D\beta^+_B,D^2\beta^+_B)
\ge f+1.
\]
Moreover,
\[
\beta^-_B-\underline z
=-B|x-x_b|^2-(C-C_0^{\eps,\nu})t\le0,
\]
\[
\beta^+_B-\overline z
=B|x-x_b|^2+(C-C_0^{\eps,\nu})t\ge0.
\]
These inequalities hold throughout the local space--time domain and hence
on its artificial parabolic-boundary portion. If
\(z_b=(0,y_b)\in\{0\}\times\partial\Omega\), the condition
\(w^0=0\) on \(\{d_\partial\le2\delta_b\}\) also gives the required
ordering on the portion contained in \((0,S)\times\partial\Omega\).

At a point $z_b=(t_b,y_b)\in(0,S)\times\partial\Omega$, choose the bounded relative parabolic domain
$D_b=(t_b-r,t_b+r)\times(\Omega\cap B_r(y_b))$ with
\[
 0<r<\min\{\delta,\delta_b,t_b/2,S-t_b\},
\]
where $\delta$ and $\kappa_0$ are from Lemma~\ref{lem:boundary-distance-control}.  Choose $B>0$ with
$Br^2\ge C_0^{\eps,\nu}S+1$.  For
\[
 q_A(t,x)=A d_\partial(x)+B\bigl(|x-y_b|^2+(t-t_b)^2\bigr),
\]
the identities \(D^2d_\partial\,Dd_\partial=0\) and
\(\nabla G=0\) give
\[
 \cL^{\eps,\nu}(x,A Dd_\partial,A D^2d_\partial)
 =A(\nu+G(x))H_{d_\partial(x)}\ge A\kappa_0.
\]
Assume that $A\ge4Br$ and define
\[
 \mathcal Q_\eps(p):=\frac{p\otimes p}{\eps^2+|p|^2}.
\]
For $q\ne0$ and $|p-q|\le |q|/2$, every vector
$r_s=q+s(p-q)$, $0\le s\le1$, satisfies $|r_s|\ge |q|/2$.  Moreover,
\[
 D\mathcal Q_\eps(r)[h]
 =\frac{h\otimes r+r\otimes h}{\eps^2+|r|^2}
 -\frac{2(r\cdot h)r\otimes r}{(\eps^2+|r|^2)^2},
 \qquad
 \|D\mathcal Q_\eps(r)[h]\|\le\frac{4|h|}{|r|}.
\]
Integration along the segment from $q$ to $p$ therefore gives
\[
 \|\mathcal Q_\eps(p)-\mathcal Q_\eps(q)\|
 \le\frac{8}{|q|}|p-q|.
\]
For $p_\pm=D(\pm q_A)$ and $q_\pm=\pm A Dd_\partial$, one has
$|p_\pm-q_\pm|\le2Br\le A/2$ and hence
\[
 \|\mathcal Q_\eps(p_\pm)-\mathcal Q_\eps(q_\pm)\|
 \le\frac{16Br}{A}.
\]
Let
\[
 K_2:=\sup_{\{0<d_\partial<\delta_b\}}\|D^2d_\partial\|,
 \qquad C_1^{\nu}:=2d\nu+2g^*(d+1),
 \qquad C_2:=16d g^*K_2.
\]
Since $D^2(\pm q_A)=\pm(A D^2d_\partial+2BI)$,
$\nabla G=0$ on $\{0<d_\partial<\delta_b\}$, and
$\|\mathcal Q_\eps(p)\|\le1$, formula~\eqref{eq:regularized-operator}
yields
\[
 \left|\cL^{\eps,\nu}(x,D(\pm q_A),D^2(\pm q_A))
 \mp\cL^{\eps,\nu}(x,A Dd_\partial,A D^2d_\partial)\right|
 \le C_1^{\nu}B+C_2Br.
\]
Choose $A$ additionally so that
\[
 A\kappa_0\ge M+2+C_1^{\nu}B+(C_2+2)Br.
\]
Set \(\beta^+_B=q_A\) and \(\beta^-_B=-q_A\). The preceding
estimates imply
\[
\partial_t\beta^-_B
+\cL^{\eps,\nu}(x,D\beta^-_B,D^2\beta^-_B)
\le f-1,
\]
\[
\partial_t\beta^+_B
+\cL^{\eps,\nu}(x,D\beta^+_B,D^2\beta^+_B)
\ge f+1.
\]
On \((0,S)\times\partial\Omega\),
\[
\beta^-_B\le0\le\beta^+_B,
\]
and the choice \(Br^2\ge C_0^{\eps,\nu}S+1\) gives
\[
\beta^-_B\le\underline z,
\qquad
\overline z\le\beta^+_B
\]
on the artificial parabolic-boundary portion. All derivatives in these
inequalities are evaluated in the original coordinates.

Choose \(0<\sigma<\rho/2\) and define
\[
\widetilde\beta^-_{b,\rho}:=\beta^-_B-\sigma,
\qquad
\widetilde\beta^+_{b,\rho}:=\beta^+_B+\sigma.
\]
Since \(\sigma\) is independent of \(t\) and \(x\),
\[
\partial_t\widetilde\beta^\pm_{b,\rho}=\partial_t\beta^\pm_B,
\qquad
D\widetilde\beta^\pm_{b,\rho}=D\beta^\pm_B,
\qquad
D^2\widetilde\beta^\pm_{b,\rho}=D^2\beta^\pm_B.
\]
Thus the differential inequalities and the ordering on
\(\Gamma_{b,\mathrm{art}}\) remain valid, while
\[
\widetilde\beta^-_{b,\rho}\le\gamma-\sigma,
\qquad
\gamma+\sigma\le\widetilde\beta^+_{b,\rho}
\quad\text{on }\Gamma_{b,\mathrm{data}}.
\]
In each of the three boundary cases constructed above,
\[
\beta^-_B(z_b)=\beta^+_B(z_b)=\gamma(z_b).
\]
Continuity therefore gives
\[
\lim_{D_b\ni z\to z_b}\widetilde\beta^-_{b,\rho}(z)
=\gamma(z_b)-\sigma>\gamma(z_b)-\rho,
\qquad
\lim_{D_b\ni z\to z_b}\widetilde\beta^+_{b,\rho}(z)
=\gamma(z_b)+\sigma<\gamma(z_b)+\rho.
\]
Hence \textup{(B1)}--\textup{(B3)} hold.
\end{proof}

\begin{proposition}[Fixed-regularization Dirichlet problem]\label{prop:regularized-dirichlet}
Fix $\eps,\nu>0$, and $T>0$.  Let
$w^0\in C^2(\overline\Omega)$ satisfy
\[
0\le w^0\le1,
\qquad
w^0\equiv0
\quad\text{on }\{x\in\overline\Omega:\dist(x,\partial\Omega)\le2\delta_b\},
\]
and let $f\in C([0,T]\times\overline\Omega)$ satisfy
\[
-M\le f\le0,
\qquad
f=0\quad\text{on }(0,T]\times\partial\Omega.
\]
Then the regularized problem associated with \eqref{eq:regularized-operator}, with initial
datum $w^0$ and homogeneous datum on $(0,T]\times\partial\Omega$, has comparison and a unique bounded continuous
\emph{Dirichlet} viscosity solution in the closed-time-interval sense of
Remark~\ref{rem:closed-time-interval}.  Moreover, the boundary-distance estimate is uniform for all
\(\eps,\nu>0\), while the initial-time estimate is uniform when \(\nu\)
ranges over a bounded subset of \((0,\infty)\).
\end{proposition}

\begin{proof}
Set \(T^+=T+1\) and write \(\widehat f:=\mathcal E_Tf\).
Lemma~\ref{lem:fixed-parameter-structural-comparison}, applied with
\(h=\widehat f\), gives comparison for the fixed-parameter problem on
\([0,T^+)\times\Omega\).

Apply Lemma~\ref{lem:regularized-local-perron-functions} on
$[0,T^+)\times\overline\Omega$ with the continued function $\widehat f$. The global and local functions constructed there satisfy \textup{(B1)}--\textup{(B3)} of Lemma~\ref{lem:perron-boundary-closure}, and continuity at $p=0$ permits the use of Lemma~\ref{lem:strict-relaxed-boundary-comparison}. Lemma~\ref{lem:perron-boundary-closure} therefore yields a continuous Dirichlet viscosity solution attaining the prescribed initial data and boundary data on $(0,T^+]\times\partial\Omega$. The comparison principle gives uniqueness.
Restricting to $[0,T]$ yields the claimed solution on the closed time interval and makes the restriction
independent of the admissible continuation.  The uniform boundary-distance and initial-time estimates are
proved separately in Lemmas~\ref{lem:regularized-boundary}--\ref{lem:initial-time-control}.
\end{proof}

\begin{proposition}[Comparison for the regularized equation with monotone reaction]
\label{prop:regularized-reaction-comparison}
Fix \(\eps,\nu>0\) and \(T>0\). Let \(u\) be a bounded
upper-semicontinuous subsolution and \(v\) a bounded lower-semicontinuous
supersolution of
\[
\partial_tz+\cL^{\eps,\nu}(x,Dz,D^2z)+\mu R_\eta(x,z)=0
\qquad\text{in }Q_T.
\]
If \(u^*\le v_*\) on \(\partial_PQ_T\), then \(u\le v\) in \(Q_T\).
\end{proposition}

\begin{proof}
Define
\[
F_\eta^{\eps,\nu}(x,r,p,X)
:=\cL^{\eps,\nu}(x,p,X)+\mu R_\eta(x,r).
\]
The operator is continuous and degenerate elliptic. If \(r\ge s\), then
\[
F_\eta^{\eps,\nu}(x,r,p,X)-F_\eta^{\eps,\nu}(x,s,p,X)
=\mu\Lambda(x)\bigl(H_\eta(r-q)-H_\eta(s-q)\bigr)\ge0,
\]
so it is proper. The matrix estimate in
Lemma~\ref{lem:fixed-parameter-structural-comparison} remains valid after
adding the reaction term because
\[
|R_\eta(x,r)-R_\eta(y,r)|
\le\|D\Lambda\|_{L^\infty(\Omega)}|x-y|.
\]
Thus condition (3.14) of \cite{CrandallIshiiLions1992} holds with the
previous modulus enlarged by a linear term. Theorem~8.2 of that reference
gives the asserted comparison.
\end{proof}

\begin{lemma}[Fixed-regularization stability under prescribed inhomogeneities]
\label{lem:regularized-inhomogeneous-stability}
Fix $\eps,\nu>0$, and $T>0$.  Let $f,g\in
C([0,T]\times\overline{\Omega})$ satisfy
\[
-M\le f,g\le0,
\qquad
f=g=0\quad\text{on }(0,T]\times\partial\Omega,
\]
and let $v_f^{\eps,\nu}$ and $v_g^{\eps,\nu}$ be the Dirichlet viscosity solutions of the
corresponding regularized equations with initial data
$v_f^0,v_g^0\in C^2(\overline\Omega)$ satisfying
\[
v_f^0=v_g^0=0
\qquad\text{on }\{x\in\overline\Omega:d_\partial(x)\le2\delta_b\},
\]
and otherwise belonging to the class specified in
Proposition~\ref{prop:regularized-dirichlet}. Then
\begin{equation}\label{eq:regularized-inhomogeneous-stability}
\norm{v_f^{\eps,\nu}(t)-v_g^{\eps,\nu}(t)}_{L^\infty}
\le \norm{v_f^0-v_g^0}_{L^\infty}
+\int_0^t\norm{f(s)-g(s)}_{L^\infty}\,\dd s.
\end{equation}
\end{lemma}

\begin{proof}
Set
\[
 a(t)=\norm{v_f^0-v_g^0}_{L^\infty}
 +\int_0^t\norm{f(s)-g(s)}_{L^\infty}\,\dd s.
\]
The map $a$ belongs to $C^1([0,T])$ and satisfies
\[
a'(t)=\|f(t)-g(t)\|_{L^\infty(\Omega)},\qquad
-a'(t)\le f(t,x)-g(t,x)\le a'(t).
\]
Because $a$ is independent of $x$,
\[
D_x(v_g^{\eps,\nu}\pm a)=D_xv_g^{\eps,\nu},\qquad
D_x^2(v_g^{\eps,\nu}\pm a)=D_x^2v_g^{\eps,\nu}.
\]
Let $\varphi\in C^{1,2}$ and suppose that
$v_g^{\eps,\nu}+a-\varphi$ has a local minimum at $(t_0,x_0)$. Then
$v_g^{\eps,\nu}-(\varphi-a)$ has a local minimum at the same point, and
the viscosity supersolution inequality for $v_g^{\eps,\nu}$ gives
\[
\partial_t\varphi(t_0,x_0)-a'(t_0)
+\cL^{\eps,\nu}\!\left(x_0,D_x\varphi(t_0,x_0),D_x^2\varphi(t_0,x_0)\right)
-g(t_0,x_0)\ge0.
\]
Consequently,
\[
\partial_t\varphi(t_0,x_0)
+\cL^{\eps,\nu}\!\left(x_0,D_x\varphi(t_0,x_0),D_x^2\varphi(t_0,x_0)\right)-f(t_0,x_0)
\ge a'(t_0)+g(t_0,x_0)-f(t_0,x_0)\ge0.
\]
Thus $v_g^{\eps,\nu}+a$ is a viscosity supersolution of the equation with
inhomogeneity $f$. Similarly, if
$v_g^{\eps,\nu}-a-\varphi$ has a local maximum at $(t_0,x_0)$, then
$v_g^{\eps,\nu}-(\varphi+a)$ has a local maximum there, and
\[
\partial_t\varphi(t_0,x_0)
+\cL^{\eps,\nu}\!\left(x_0,D_x\varphi(t_0,x_0),D_x^2\varphi(t_0,x_0)\right)-f(t_0,x_0)
\le g(t_0,x_0)-f(t_0,x_0)-a'(t_0)\le0.
\]
Hence $v_g^{\eps,\nu}-a$ is a viscosity subsolution of the equation with
inhomogeneity $f$.

The ordered boundary hypotheses of the fixed-regularization comparison principle
are again explicit.  On $(0,T]\times\partial\Omega$,
\[
 (v_g^{\eps,\nu}-a)^*=-a\le0=(v_f^{\eps,\nu})_*,
 \qquad
 (v_f^{\eps,\nu})^*=0\le a=(v_g^{\eps,\nu}+a)_*,
\]
while on $\{0\}\times\overline\Omega$,
\[
 (v_g^{\eps,\nu}-a)^*=v_g^0-a(0)\le v_f^0=(v_f^{\eps,\nu})_*,
 \qquad
 (v_f^{\eps,\nu})^*=v_f^0\le v_g^0+a(0)
 =(v_g^{\eps,\nu}+a)_*.
\]
Comparison from Proposition~\ref{prop:regularized-dirichlet} therefore yields
\[
 v_g^{\eps,\nu}-a\le v_f^{\eps,\nu}
 \le v_g^{\eps,\nu}+a,
\]
which is \eqref{eq:regularized-inhomogeneous-stability}.
\end{proof}

\begin{lemma}[Boundary-distance estimate uniform in the regularization]
\label{lem:regularized-boundary}
Under Assumption~\ref{ass:data}, there exist constants
\[
\delta\in(0,\delta_b]
\qquad\text{and}\qquad
\kappa_b>0
\]
such that the following assertion holds. Let \(\eps,\nu>0\),
\(T,M>0\), and let \(w^{\eps,\nu}\) be the solution from
Proposition~\ref{prop:regularized-dirichlet} corresponding to data satisfying
\[
-M\le f\le0
\quad\text{in }[0,T]\times\overline\Omega,
\qquad
f=0
\quad\text{on }(0,T]\times\partial\Omega,
\]
\[
0\le w^0\le1
\quad\text{in }\overline\Omega,
\qquad
w^0(x)=0
\quad\text{whenever }d_\partial(x)\le2\delta_b.
\]
Then there exists \(A_{M,T}>0\), independent of
\((\eps,\nu)\) and of the solution, such that
\[
-A_{M,T}d_\partial(x)
\le w^{\eps,\nu}(t,x)
\le A_{M,T}d_\partial(x)
\]
for every
\[
(t,x)\in[0,T]\times\Omega
\qquad\text{with}\qquad
0<d_\partial(x)<\delta.
\]
\end{lemma}

\begin{proof}
Let \(\delta\) be the constant from
Lemma~\ref{lem:boundary-distance-control}. For
\[
0<s<\delta,
\]
the curvature computation in that lemma gives
\[
H_s\ge\frac{H_0}{2}.
\]
Moreover,
\[
DG=0
\qquad\text{on }
\{x\in\Omega:0<d_\partial(x)<\delta_b\}.
\]
For
\[
p=A Dd_\partial,
\qquad
X=A D^2d_\partial,
\]
one has
\[
\operatorname{tr}X=-AH_s,
\qquad
Xp=A^2D^2d_\partial\,Dd_\partial=0.
\]
Therefore,
\begin{align*}
\cL^{\eps,\nu}(x,p,X)
&=
-\nu\operatorname{tr}X
-G(x)\left(
\operatorname{tr}X-
\frac{Xp\cdot p}{\eps^2+|p|^2}
\right)
-DG(x)\cdot p\\
&=
A(\nu+G(x))H_s\\
&\ge
A g_*\frac{H_0}{2}.
\end{align*}
Replacing \((p,X)\) by \((-p,-X)\) changes the sign of the
operator value. Set
\[
\kappa_b:=\frac{g_*H_0}{2}.
\]
Then
\[
\cL^{\eps,\nu}
\bigl(x,A Dd_\partial,A D^2d_\partial\bigr)
\ge A\kappa_b
\]
and
\[
\cL^{\eps,\nu}
\bigl(x,-A Dd_\partial,-A D^2d_\partial\bigr)
\le -A\kappa_b
\]
for every \(\eps,\nu>0\) and every
\(x\) satisfying \(0<d_\partial(x)<\delta\).

For \(\sigma>0\), define
\[
\overline q_\sigma(t,x)=1+\sigma(1+t),
\qquad
\underline q_\sigma(t,x)=-Mt-\sigma(1+t).
\]
Since
\[
D\overline q_\sigma=D\underline q_\sigma=0,
\qquad
D^2\overline q_\sigma=D^2\underline q_\sigma=0,
\]
one has
\[
\partial_t\overline q_\sigma
+\cL^{\eps,\nu}
\bigl(x,D\overline q_\sigma,D^2\overline q_\sigma\bigr)
=\sigma
\ge f+\sigma,
\]
and
\[
\partial_t\underline q_\sigma
+\cL^{\eps,\nu}
\bigl(x,D\underline q_\sigma,D^2\underline q_\sigma\bigr)
=-M-\sigma
\le f-\sigma.
\]
The initial inequalities and the inequalities on
\((0,T]\times\partial\Omega\) required by
Lemma~\ref{lem:strict-relaxed-boundary-comparison} also hold. Hence
\[
-Mt-\sigma(1+t)
\le w^{\eps,\nu}(t,x)
\le1+\sigma(1+t)
\]
on \([0,T]\times\overline\Omega\). Taking
\(\sigma=1/n\), \(n\in\mathbb N\), gives
\begin{equation}
\label{eq:regularized-global-range}
-Mt\le w^{\eps,\nu}(t,x)\le1
\qquad
\text{on }[0,T]\times\overline\Omega.
\end{equation}

Choose \(A_{M,T}>0\) such that
\[
A_{M,T}\delta\ge1+MT,
\qquad
A_{M,T}\kappa_b\ge M.
\]
Set
\[
D_\delta
=
(0,T)\times
\{x\in\Omega:0<d_\partial(x)<\delta\},
\]
and define
\[
\overline b_\sigma(t,x)
=
A_{M,T}d_\partial(x)+\sigma(1+t),
\]
\[
\underline b_\sigma(t,x)
=
-A_{M,T}d_\partial(x)-\sigma(1+t).
\]
The preceding operator inequalities give
\[
\partial_t\overline b_\sigma
+
\cL^{\eps,\nu}
\bigl(x,D\overline b_\sigma,D^2\overline b_\sigma\bigr)
\ge
A_{M,T}\kappa_b+\sigma
\ge f+\sigma,
\]
and
\[
\partial_t\underline b_\sigma
+
\cL^{\eps,\nu}
\bigl(x,D\underline b_\sigma,D^2\underline b_\sigma\bigr)
\le
-A_{M,T}\kappa_b-\sigma
\le f-\sigma
\]
in \(D_\delta\).

On
\[
\partial_PD_\delta\cap
\left(
\{0\}\times\overline\Omega
\ \cup\
(0,T)\times\partial\Omega
\right),
\]
the assumptions on \(w^0\) and the homogeneous boundary datum give
\[
\underline b_\sigma\le\gamma-\sigma,
\qquad
\gamma+\sigma\le\overline b_\sigma.
\]
On
\[
(0,T)\times\{x\in\Omega:d_\partial(x)=\delta\},
\]
the estimate \eqref{eq:regularized-global-range} and
\(A_{M,T}\delta\ge1+MT\) give
\[
\underline b_\sigma
\le w^{\eps,\nu}
\le\overline b_\sigma.
\]
Lemma~\ref{lem:strict-relaxed-boundary-comparison} yields
\[
-A_{M,T}d_\partial(x)-\sigma(1+t)
\le w^{\eps,\nu}(t,x)
\le
A_{M,T}d_\partial(x)+\sigma(1+t)
\]
for every \((t,x)\in D_\delta\). Taking
\(\sigma=1/n\), \(n\in\mathbb N\), proves
\[
-A_{M,T}d_\partial(x)
\le w^{\eps,\nu}(t,x)
\le A_{M,T}d_\partial(x).
\]
All constants are independent of \(\eps,\nu>0\).
\end{proof}

\begin{lemma}[Uniform initial-time estimate for the regularized equation]\label{lem:initial-time-control}
Let $w^0\in C^2(\overline{\Omega})$ satisfy the condition $w^0=0$ on $\{d_\partial\le2\delta_b\}$ and let
$\norm{f}_{L^\infty([0,T]\times\Omega)}\le M$.  For every \(\overline\nu>0\), there is a constant \(C_0\), depending on
\(w^0\), \(M\), the fixed coefficients, and \(\overline\nu\), but independent
of \(\eps>0\) and \(0<\nu\le\overline\nu\), such that every
continuous regularized Dirichlet viscosity solution attaining the homogeneous data on
$(0,T)\times\partial\Omega$ satisfies
\[
\norm{w^{\eps,\nu}(t,\cdot)-w^0}_{L^\infty(\Omega)}\le C_0t
\qquad(0\le t\le T).
\]
\end{lemma}

\begin{proof}
Since
\[
\left\|\frac{p\otimes p}{\eps^2+|p|^2}\right\|\le1
\qquad(p\in\mathbb R^d),
\]
there exists a finite constant
\[
B_0
:=
\sup_{\substack{\eps>0\\0<\nu\le\overline\nu}}
\left\|
\cL^{\eps,\nu}(x,Dw^0,D^2w^0)
\right\|_{L^\infty(\Omega)}<\infty.
\]
In particular, if $Dw^0(x)=0$, then
\[
\cL^{\eps,\nu}(x,0,D^2w^0)
=-(\nu+G(x))\operatorname{tr}(D^2w^0),
\]
and this quantity is bounded in absolute value by $B_0$, uniformly for \(\eps>0\) and \(0<\nu\le\overline\nu\). Define
\[
C_0:=B_0+M+1.
\]
For $\sigma>0$, set
\[
\overline b_\sigma(t,x):=w^0(x)+C_0t+\sigma(1+t),
\qquad
\underline b_\sigma(t,x):=w^0(x)-C_0t-\sigma(1+t).
\]
Because the added terms depend only on $t$,
\[
D_x\overline b_\sigma=D_x\underline b_\sigma=Dw^0,
\qquad
D_x^2\overline b_\sigma=D_x^2\underline b_\sigma=D^2w^0,
\]
whereas
\[
\partial_t\overline b_\sigma=C_0+\sigma,
\qquad
\partial_t\underline b_\sigma=-C_0-\sigma.
\]
Consequently, using $|f|\le M$ and the definition of $B_0$,
\begin{align*}
\partial_t\overline b_\sigma
+\cL^{\eps,\nu}(x,D_x\overline b_\sigma,D_x^2\overline b_\sigma)-f
&\ge C_0+\sigma-B_0-M=1+\sigma,\\
\partial_t\underline b_\sigma
+\cL^{\eps,\nu}(x,D_x\underline b_\sigma,D_x^2\underline b_\sigma)-f
&\le -C_0-\sigma+B_0+M=-1-\sigma.
\end{align*}
Thus $\overline b_\sigma$ is a strict supersolution and
$\underline b_\sigma$ is a strict subsolution. At $t=0$,
\[
\underline b_\sigma(0,x)=w^0(x)-\sigma
\le w^0(x)\le
w^0(x)+\sigma=\overline b_\sigma(0,x).
\]
Moreover, since $w^0=0$ on $\partial\Omega$,
\[
\underline b_\sigma(t,y)=-C_0t-\sigma(1+t)
\le0\le C_0t+\sigma(1+t)=\overline b_\sigma(t,y)
\]
for every $(t,y)\in(0,T]\times\partial\Omega$. Comparison therefore yields
\[
w^0(x)-C_0t-\sigma(1+t)
\le w^{\eps,\nu}(t,x)
\le w^0(x)+C_0t+\sigma(1+t).
\]
Setting $\sigma=1/n$ for every $n\in\mathbb N$ and using $1/n\to0$ gives
\[
\|w^{\eps,\nu}(t,\cdot)-w^0\|_{L^\infty(\Omega)}\le C_0t.
\]
\end{proof}

\begin{proof}[Proof of Proposition~\ref{prop:direct-zero-gradient-limits}]
For fixed \((x,X)\), define
\[
\mathcal A_{x,X}
=
\left\{
-G(x)\operatorname{tr}X+G(x)r\,e\cdot Xe:
0\le r\le1,
\ e\in\mathbb S^{d-1}
\right\}.
\]
Let
\[
(\eps_k,\nu_k,y_k,p_k,Y_k)
\to(0,0,x,0,X).
\]
If \(p_k\neq0\), write \(p_k=|p_k|e_k\) and set
\[
r_k=\frac{|p_k|^2}{\eps_k^2+|p_k|^2};
\]
if \(p_k=0\), set \(r_k=0\). After extraction of a subsequence,
\(r_k\to r\in[0,1]\); if \(r>0\), a further subsequence satisfies
\(e_k\to e\in\mathbb S^{d-1}\). Since
\[
G(y_k)\to G(x),
\qquad
Y_k\to X,
\qquad
\nu_k\operatorname{tr}Y_k\to0,
\qquad
DG(y_k)\cdot p_k\to0,
\]
every subsequential limit of
\(\cL^{\eps_k,\nu_k}(y_k,p_k,Y_k)\) belongs to
\(\mathcal A_{x,X}\). Therefore
\[
\limsup_{(\eps,\nu,y,p,Y)\to(0,0,x,0,X)}
\cL^{\eps,\nu}(y,p,Y)
\le\max\mathcal A_{x,X},
\]
\[
\liminf_{(\eps,\nu,y,p,Y)\to(0,0,x,0,X)}
\cL^{\eps,\nu}(y,p,Y)
\ge\min\mathcal A_{x,X}.
\]

Conversely, fix \(r_0\in[0,1]\) and \(e\in\mathbb S^{d-1}\), and take
\(y_k=x\), \(Y_k=X\), and \(\nu_k=k^{-1}\). If \(0<r_0<1\), set
\[
\eps_k=k^{-1},
\qquad
p_k=\eps_k\sqrt{\frac{r_0}{1-r_0}}\,e.
\]
If \(r_0=0\), set \(\eps_k=k^{-1}\) and \(p_k=\eps_k^2e\). If
\(r_0=1\), set \(\eps_k=k^{-2}\) and \(p_k=k^{-1}e\). In each case,
\[
\frac{|p_k|^2}{\eps_k^2+|p_k|^2}\to r_0,
\]
and the corresponding operator values converge to
\[
-G(x)\operatorname{tr}X+G(x)r_0\,e\cdot Xe.
\]
Thus every element of \(\mathcal A_{x,X}\) is a joint limiting value, and
the upper and lower limits are \(\max\mathcal A_{x,X}\) and
\(\min\mathcal A_{x,X}\), respectively. By the Rayleigh quotient
characterization of the extremal eigenvalues,
\[
\max_{e\in\mathbb S^{d-1}}e\cdot Xe=\lambda_{\max}(X),
\qquad
\min_{e\in\mathbb S^{d-1}}e\cdot Xe=\lambda_{\min}(X).
\]
Since \(G(x)>0\), it follows that
\begin{align*}
\max\mathcal A_{x,X}
&=-G(x)\operatorname{tr}X
  +G(x)\max_{0\le r\le1}r\lambda_{\max}(X)\\
&=-G(x)\operatorname{tr}X
  +G(x)\max\{0,\lambda_{\max}(X)\},\\
\min\mathcal A_{x,X}
&=-G(x)\operatorname{tr}X
  +G(x)\min_{0\le r\le1}r\lambda_{\min}(X)\\
&=-G(x)\operatorname{tr}X
  +G(x)\min\{0,\lambda_{\min}(X)\}.
\end{align*}
These identities are precisely \eqref{eq:raw-lower}--\eqref{eq:raw-upper}.
Finally,
\[
\min\{0,\lambda_{\min}(X)\}
\le\lambda_{\min}(X)
\le\lambda_{\max}(X)
\le\max\{0,\lambda_{\max}(X)\},
\]
which proves \eqref{eq:limiting-value-order}.
\end{proof}

\begin{remark}[Why direct operator stability is insufficient]\label{rem:raw-not-selection}
The direct upper and lower limiting values in Proposition~\ref{prop:direct-zero-gradient-limits} need not coincide with the geometric limiting extensions in
Definition~\ref{def:viscosity}. The discrepancy depends on the Hessian: the first inequality in
\eqref{eq:limiting-value-order} is strict when $X$ is positive definite, while the last is strict
when $X$ is negative definite. If
$\lambda_{\min}(X)\le0\le\lambda_{\max}(X)$, including $X=0$, the direct and geometric
bounds agree. For a local extremum associated with a smooth function $\varphi$ satisfying
$D_x\varphi=0$ and with positive- or negative-definite $D_x^2\varphi$, direct application of the
half-relaxed-limit inequalities produces $\cJlower(x,0,D_x^2\varphi)$ or
$\cJupper(x,0,D_x^2\varphi)$, rather than $\cLgeoLower(x,0,D_x^2\varphi)$ or
$\cLgeoUpper(x,0,D_x^2\varphi)$. The convergence proof therefore establishes the inequalities
specified in Definition~\ref{def:vanishing-derivative-formulation} and then applies
Proposition~\ref{prop:vanishing-derivative-equivalence} to obtain the geometric viscosity
inequalities for every symmetric Hessian.
\end{remark}

\begin{lemma}[Uniform small-derivative control for the regularization]
\label{lem:small-derivative-control}
For every \(\overline\nu>0\), there is a constant \(C_{\overline\nu}\),
independent of \(\eps>0\) and \(0<\nu\le\overline\nu\), such that
\begin{equation}\label{eq:uniform-small-derivative-bound}
\left|\cL^{\eps,\nu}(x,p,X)\right|
\le C_{\overline\nu}\bigl(|p|+\|X\|\bigr)
\qquad
(x,p,X)\in\overline\Omega\times\R^d\times\mathbb S^d.
\end{equation}
Consequently, if $(x_n,p_n,X_n)\to(x_0,0,0)$ and
$(\eps_n,\nu_n)$ satisfies $0<\eps_n,\nu_n$ and
$(\eps_n,\nu_n)\to(0,0)$, then
\[
\cL^{\eps_n,\nu_n}(x_n,p_n,X_n)\longrightarrow0.
\]
\end{lemma}

\begin{proof}
Define
\[
Q_\eps(p):=\frac{p\otimes p}{\eps^2+|p|^2}.
\]
For every \(\xi\in\mathbb R^d\),
\[
|Q_\eps(p)\xi|
=\frac{|p|\,|p\cdot\xi|}{\eps^2+|p|^2}
\le\frac{|p|^2}{\eps^2+|p|^2}|\xi|
\le|\xi|.
\]
Consequently,
\[
\|Q_\eps(p)\|_{\mathrm{op}}
=\frac{|p|^2}{\eps^2+|p|^2}\le1,
\qquad
\left|\frac{Xp\cdot p}{\eps^2+|p|^2}\right|
=|X:Q_\eps(p)|\le\|X\|.
\]
Using \(0<\nu\le\overline\nu\), \(|\operatorname{tr}X|\le d\|X\|\), and
\(0<G(x)\le g^*\) in \eqref{eq:regularized-operator}, we obtain
\[
\left|\cL^{\eps,\nu}(x,p,X)\right|
\le \bigl(d\overline\nu+(d+1)g^*\bigr)\|X\|
+\|\nabla G\|_{L^\infty}|p|.
\]
Thus \eqref{eq:uniform-small-derivative-bound} holds. For a sequence \(\nu_n\to0\), choose
\(\overline\nu\ge\sup_n\nu_n\). If
\((x_n,p_n,X_n)\to(x_0,0,0)\), the right-hand side converges to zero;
hence
\[
\cL^{\eps_n,\nu_n}(x_n,p_n,X_n)\longrightarrow0.
\]
\end{proof}

\begin{lemma}[Equivalence of the geometric and vanishing-derivative formulations]
\label{lem:continuous-inhomogeneity-vanishing-derivative-equivalence}
Let \(\mathfrak G(x,p,X)\) be continuous for \(p\neq0\) and degenerate
elliptic in the sense that
\[
X\ge Y
\quad\Longrightarrow\quad
\mathfrak G(x,p,X)\le\mathfrak G(x,p,Y)
\qquad(p\neq0).
\]
Assume
\begin{equation}\label{eq:principal-small-derivative}
\lim_{(p,X)\to(0,0)}
\sup_{x\in\overline\Omega}|\mathfrak G(x,p,X)|=0,
\end{equation}
and suppose that
\[
\mathfrak G_*(x,0,X)
=
\liminf_{\substack{(y,p,Y)\to(x,0,X)\\p\neq0}}
\mathfrak G(y,p,Y),
\qquad
\mathfrak G^*(x,0,X)
=
\limsup_{\substack{(y,p,Y)\to(x,0,X)\\p\neq0}}
\mathfrak G(y,p,Y)
\]
are finite for every \((x,X)\in\overline\Omega\times\mathbb S^d\). Let
\(a\in C([0,T]\times\overline\Omega)\). Then the geometric viscosity
formulation of
\[
\partial_t v+\mathfrak G(x,Dv,D^2v)=a(t,x)
\qquad\text{in }Q_T
\]
is equivalent to the following vanishing-derivative formulation. If
\(v-\varphi\) has a local maximum at \(z_0\), impose
\[
\partial_t\varphi(z_0)
+\mathfrak G\bigl(x_0,D_x\varphi(z_0),D_x^2\varphi(z_0)\bigr)
\le a(z_0)
\]
when \(D_x\varphi(z_0)\neq0\), and impose
\[
\partial_t\varphi(z_0)\le a(z_0)
\]
when
\(D_x\varphi(z_0)=0\) and \(D_x^2\varphi(z_0)=0\). The supersolution
conditions are obtained by replacing local maxima by local minima and
reversing the inequalities.
\end{lemma}

\begin{proof}
Condition \eqref{eq:principal-small-derivative} implies
\[
\mathfrak G_*(x,0,0)=\mathfrak G^*(x,0,0)=0.
\]
Hence every geometric viscosity subsolution or supersolution satisfies the conditions in Definition~\ref{def:vanishing-derivative-formulation}.

Conversely, let \(u\) satisfy the subsolution conditions in Definition~\ref{def:vanishing-derivative-formulation}, and let
\(\varphi\in C^{1,2}\) be such that \(u-\varphi\) has a local maximum at
\(z_0=(t_0,x_0)\), with
\[
D_x\varphi(z_0)=0,
\qquad
X_0=D_x^2\varphi(z_0).
\]
After replacing \(\varphi\) by
\(\varphi+|x-x_0|^4+|t-t_0|^2\) with an arbitrarily small coefficient, the
maximum may be assumed strict. For \(\delta>0\), maximize
\[
\Psi_\delta(x,y,t)
=
u(x,t)-\frac{|x-y|^4}{4\delta}-\varphi(y,t)
\]
in a compact neighborhood of \((x_0,x_0,t_0)\), and denote a maximizing
triple by \((x_\delta,y_\delta,t_\delta)\). Then
\[
(x_\delta,y_\delta,t_\delta)\to(x_0,x_0,t_0).
\]
The first- and second-order necessary conditions with respect to \(y\) are
\begin{align}
-\frac{|x_\delta-y_\delta|^2}{\delta}(x_\delta-y_\delta)
+D_x\varphi(y_\delta,t_\delta)&=0,
\label{eq:caseB-gradient}\\
\frac{|x_\delta-y_\delta|^2}{\delta}I
+\frac{2}{\delta}(x_\delta-y_\delta)\otimes(x_\delta-y_\delta)
+D_x^2\varphi(y_\delta,t_\delta)&\ge0.
\label{eq:caseB-hessian}
\end{align}
Moreover,
\[
\chi_\delta(x,t)
=
\frac{|x_\delta-y_\delta|^4}{4\delta}
+
\varphi\bigl(x-(x_\delta-y_\delta),t\bigr)
\]
satisfies that \(u-\chi_\delta\) has a local maximum at
\((x_\delta,t_\delta)\).

If \(D_x\varphi(y_\delta,t_\delta)\neq0\) along a subsequence, the nonzero-gradient inequality in Definition~\ref{def:vanishing-derivative-formulation} applied to \(\chi_\delta\) gives
\[
\partial_t\varphi(y_\delta,t_\delta)
+
\mathfrak G\bigl(x_\delta,D_x\varphi(y_\delta,t_\delta),
D_x^2\varphi(y_\delta,t_\delta)\bigr)
\le a(t_\delta,x_\delta).
\]
Taking the lower limit yields
\[
\partial_t\varphi(z_0)+\mathfrak G_*(x_0,0,X_0)
\le a(z_0).
\]

Otherwise, along a subsequence,
\(D_x\varphi(y_\delta,t_\delta)=0\). Equation
\eqref{eq:caseB-gradient} then implies \(x_\delta=y_\delta\). The function
\[
\psi_\delta(x,t)
=
\frac{|x-y_\delta|^4}{4\delta}+
\varphi(y_\delta,t)
\]
satisfies that \(u-\psi_\delta\) has a local maximum at
\((x_\delta,t_\delta)\), and its spatial gradient and Hessian vanish there.
The zero-gradient condition in Definition~\ref{def:vanishing-derivative-formulation} therefore gives
\[
\partial_t\varphi(y_\delta,t_\delta)
\le a(t_\delta,x_\delta).
\]
In addition, \eqref{eq:caseB-hessian} gives
\(D_x^2\varphi(y_\delta,t_\delta)\ge0\), and hence \(X_0\ge0\). Degenerate
ellipticity passes to the lower limiting extension, so
\[
\mathfrak G_*(x_0,0,X_0)
\le
\mathfrak G_*(x_0,0,0)=0.
\]
The convergences
\[
(y_\delta,t_\delta)\to z_0,
\qquad
(t_\delta,x_\delta)\to z_0,
\]
together with the continuity of \(a\) and \(\partial_t\varphi\), imply
\[
\partial_t\varphi(z_0)\le a(z_0).
\]
Combining this inequality with
\(\mathfrak G_*(x_0,0,X_0)\le0\) gives
\[
\partial_t\varphi(z_0)+\mathfrak G_*(x_0,0,X_0)
\le a(z_0).
\]
Every sequence \(\delta_k\downarrow0\) has a subsequence of one of these
two types; therefore the geometric subsolution inequality holds.

For supersolutions, apply the subsolution result to \(\widetilde u=-u\) and
\[
\widetilde{\mathfrak G}(x,p,X)
=-\mathfrak G(x,-p,-X).
\]
Since
\[
\widetilde{\mathfrak G}_*(x,0,-X)
=-\mathfrak G^*(x,0,X),
\]
the transformed subsolution inequality is exactly the geometric
supersolution inequality for \(u\).
\end{proof}

\begin{proof}[Proof of Proposition~\ref{prop:vanishing-derivative-equivalence}]
Apply Lemma~\ref{lem:continuous-inhomogeneity-vanishing-derivative-equivalence} with
$\mathfrak G=\cL$ and $a=f$.  Lemma~\ref{lem:structural-comparison} gives continuity
away from $p=0$ and degenerate ellipticity.  Its estimate
\[
 |\cL(x,p,X)|\le(d+1)g^*\|X\|
 +\|\nabla G\|_{L^\infty}|p|
\]
is uniform in $x$ and verifies \eqref{eq:principal-small-derivative}; the finite geometric limiting extensions at zero gradient
are exactly \eqref{eq:Lsubstar}--\eqref{eq:Lstar}.  Thus the vanishing-derivative and geometric viscosity formulations agree for the spatially inhomogeneous weighted
operator with an arbitrary continuous inhomogeneous term.
\end{proof}

\begin{remark}[Vanishing gradient with nonzero Hessian]
\label{rem:nonzero-hessian-zero-gradient-coverage}
Let $u-\varphi$ have a local maximum at $z_0$ and suppose that
\[
D_x\varphi(z_0)=0,
\qquad
D_x^2\varphi(z_0)\ne0.
\]
Let $(x_\delta,y_\delta,t_\delta)$ be the maximizing triples introduced in the proof of
Lemma~\ref{lem:continuous-inhomogeneity-vanishing-derivative-equivalence}. If
$D_x\varphi(y_\delta,t_\delta)\neq0$ along a sequence $\delta_k\downarrow0$, then the
corresponding subsolution inequalities and the definition of $\mathfrak G_*$ give the geometric
subsolution inequality at $z_0$. If
$D_x\varphi(y_\delta,t_\delta)=0$ along a sequence $\delta_k\downarrow0$, then
\eqref{eq:caseB-gradient} gives $x_\delta=y_\delta$, while
\eqref{eq:caseB-hessian} gives $D_x^2\varphi(y_\delta,t_\delta)\ge0$. Hence
$D_x^2\varphi(z_0)\ge0$, and degenerate ellipticity gives
\[
\mathfrak G_*(x_0,0,D_x^2\varphi(z_0))
\le \mathfrak G_*(x_0,0,0)=0.
\]
The supersolution statement follows after replacing local maxima by local minima and
$\mathfrak G_*$ by $\mathfrak G^*$. Consequently,
\eqref{eq:principal-small-derivative} determines the operator at $(p,X)=(0,0)$,
whereas \eqref{eq:Lsubstar}--\eqref{eq:Lstar} determine the viscosity inequalities
when $p=0$ and $X\ne0$.
\end{remark}

\begin{lemma}[Stability with separate treatment of the principal and inhomogeneous terms]
\label{lem:separate-term-vanishing-derivative-stability}
Fix $T>0$.  Let $(\eps_n,\nu_n)$ satisfy
$0<\eps_n,\nu_n$ and $(\eps_n,\nu_n)\to(0,0)$, and let
$f_n\to f$ uniformly on $[0,T]\times\overline\Omega$.  Assume that the
sequence is locally uniformly bounded:
\[
 \sup_n\|w_n\|_{L^\infty(K)}<\infty
 \qquad\text{for every compact }K\subset Q_T.
\]
If $w_n$ are viscosity subsolutions of
\[
\partial_t w_n+\cL^{\eps_n,\nu_n}(x,Dw_n,D^2w_n)=f_n,
\]
then $\overline w=\limsup^*w_n$ is a subsolution in the vanishing-derivative formulation for the limiting
geometric equation.  If $w_n$ are viscosity supersolutions, then
$\underline w=\liminf_*w_n$ is a supersolution in the vanishing-derivative formulation.
\end{lemma}

\begin{proof}
It suffices to prove the subsolution assertion. The identity
\[
\cL^{\eps,\nu}(x,-p,-X)=-\cL^{\eps,\nu}(x,p,X)
\]
shows that \(-w_n\) is a subsolution with inhomogeneity \(-f_n\) whenever
\(w_n\) is a supersolution with inhomogeneity \(f_n\). Applying the
subsolution result to \(-w_n\) therefore yields the supersolution assertion.

Let \(O\subset Q_T\) be open, with \(\overline O\) compact and
\(\overline O\subset Q_T\). Let \(z_0=(t_0,x_0)\in O\), and suppose that
\(\overline w-\varphi\) has a local maximum at \(z_0\), where
\(\varphi\in C^{1,2}(O)\) satisfies one of the two alternatives in the
vanishing-derivative formulation. Choose a parabolic neighborhood \(\mathcal P\) of \(z_0\) such that
\(\overline{\mathcal P}\) is compact and \(\overline{\mathcal P}\subset O\).
For \(\delta>0\), define
\[
\varphi_\delta(t,x)
=
\varphi(t,x)+\delta\bigl(|x-x_0|^4+(t-t_0)^2\bigr).
\]
After reducing \(\mathcal P\), the function
\(\overline w-\varphi_\delta\) has a strict maximum at \(z_0\) relative
to \(\overline{\mathcal P}\). The maximum-point lemma yields
\(z_{n,\delta}\in\mathcal P\) and constants \(c_{n,\delta}\) such that
\[
w_n-(\varphi_\delta+c_{n,\delta})
\]
has a local maximum at \(z_{n,\delta}\), and
\(z_{n,\delta}\to z_0\) as \(n\to\infty\), for fixed \(\delta\).

Assume first that \(D_x\varphi(z_0)\neq0\). Set
\(\rho=|D_x\varphi(z_0)|/2\). For all sufficiently small \(\delta\) and
large \(n\),
\[
|D_x\varphi_\delta(z_{n,\delta})|\ge\rho,
\]
and the Hessians \(D_x^2\varphi_\delta(z_{n,\delta})\) remain bounded.
For fixed \(\delta>0\), the viscosity inequality at
\(z_{n,\delta}\), Lemma~\ref{lem:away-zero-gradient-convergence}, and the
uniform convergence \(f_n\to f\) imply
\[
\partial_t\varphi_\delta(z_0)
+
\cL\bigl(x_0,D_x\varphi_\delta(z_0),
D_x^2\varphi_\delta(z_0)\bigr)
\le f(z_0).
\]
Because
\[
\partial_t\varphi_\delta(z_0)=\partial_t\varphi(z_0),
\qquad
D_x\varphi_\delta(z_0)=D_x\varphi(z_0),
\qquad
D_x^2\varphi_\delta(z_0)=D_x^2\varphi(z_0),
\]
this is exactly
\[
\partial_t\varphi(z_0)
+
\cL\bigl(x_0,D_x\varphi(z_0),D_x^2\varphi(z_0)\bigr)
\le f(z_0).
\]

Assume next that
\[
D_x\varphi(z_0)=0,
\qquad
D_x^2\varphi(z_0)=0.
\]
For fixed \(\delta\),
\[
D_x\varphi_\delta(z_{n,\delta})\to0,
\qquad
D_x^2\varphi_\delta(z_{n,\delta})
\to D_x^2\varphi_\delta(z_0)=0
\]
as \(n\to\infty\). Lemma~\ref{lem:small-derivative-control} implies
\[
\cL^{\eps_n,\nu_n}
\bigl(x_{n,\delta},D_x\varphi_\delta(z_{n,\delta}),
D_x^2\varphi_\delta(z_{n,\delta})\bigr)
\longrightarrow0.
\]
For fixed \(\delta>0\), the viscosity inequality for \(w_n\), the
convergence of the regularized operator term to zero, and
\(f_n(z_{n,\delta})\to f(z_0)\) give
\[
\partial_t\varphi_\delta(z_0)\le f(z_0).
\]
Since
\(\partial_t\varphi_\delta(z_0)=\partial_t\varphi(z_0)\),
\[
\partial_t\varphi(z_0)\le f(z_0).
\]
The two alternatives prove the reduced subsolution inequalities.
\end{proof}

\begin{lemma}[Vanishing-regularization selection for the present model]\label{lem:vanishing-regularization-selection}
Let $(\eps_n,\nu_n)$ satisfy $0<\eps_n,\nu_n$ and
$(\eps_n,\nu_n)\to(0,0)$, let $f_n,f\in C([0,T]\times\overline\Omega)$ satisfy
\[
-M\le f_n,f\le0,\qquad f_n=f=0\ \text{on }(0,T]\times\partial\Omega,
\qquad \norm{f_n-f}_{L^\infty}\longrightarrow0,
\]
and let $w_n$ be the Dirichlet viscosity solution of
\begin{equation}\label{eq:regularized-inhomogeneous}
\begin{aligned}
\partial_tw_n
-\nu_n\Delta w_n
-A_{\eps_n}(\nabla w_n)\operatorname{div}\!\left(
G\frac{\nabla w_n}{A_{\eps_n}(\nabla w_n)}\right)
&=f_n&&\text{in }(0,T)\times\Omega,\\
w_n&=0&&\text{on }(0,T)\times\partial\Omega,\\
w_n(0,\cdot)&=w^0&&\text{in }\overline\Omega,
\end{aligned}
\end{equation}
where \(w^0\in C^2(\overline\Omega)\) satisfies
\[
0\le w^0\le1,
\qquad
w^0(x)=0
\quad\text{whenever }d_\partial(x)\le2\delta_b.
\]
Then \(w_n\to w\) uniformly on
$[0,T]\times\overline\Omega$, where $w$ is the unique Dirichlet viscosity solution of
\begin{equation}\label{eq:limit-inhomogeneous}
\begin{aligned}
\partial_tw-\cK(x,\nabla w,D^2w)&=f&&\text{in }(0,T)\times\Omega,\\
w&=0&&\text{on }(0,T)\times\partial\Omega,\\
w(0,\cdot)&=w^0&&\text{in }\overline\Omega.
\end{aligned}
\end{equation}
\end{lemma}

\begin{proof}
Since \((\eps_n,\nu_n)\to(0,0)\), both parameter sequences are bounded.
Set
\[
\overline\nu:=\sup_{n\in\mathbb N}\nu_n<\infty.
\]
The boundary-distance estimate is uniform for all positive parameters, and
the initial-time and small-derivative estimates apply uniformly with this
choice of \(\overline\nu\).
Set $T^+:=T+1$ and extend $f_n$ and $f$ by constant continuation:
\[
 \widehat f_n(t,x):=
 \begin{cases}
 f_n(t,x),&0\le t\le T,\\
 f_n(T,x),&T<t\le T^+,
 \end{cases}
 \qquad
 \widehat f(t,x):=
 \begin{cases}
 f(t,x),&0\le t\le T,\\
 f(T,x),&T<t\le T^+.
 \end{cases}
\]
The extended functions take values in $[-M,0]$.  Because the hypothesis $f_n=f=0$ on
$(0,T]\times\partial\Omega$ includes $t=T$, the constant continuations retain the homogeneous value on $(0,T^+]\times\partial\Omega$.  They also satisfy
$\|\widehat f_n-\widehat f\|_{L^\infty([0,T^+]\times\overline\Omega)}\to0$.
Let $\widehat w_n$ and $\widehat w$ be the corresponding Dirichlet viscosity solutions on the
open-terminal space--time domain $[0,T^+)\times\overline\Omega$.  The restriction
$\widehat w_n|_{[0,T]\times\overline\Omega}$ and $w_n$ solve the same fixed-parameter
regularized equation on $Q_T$, with the same function $f_n$, the same initial datum, and the same zero boundary datum on $(0,T)\times\partial\Omega$;
their equality follows from the uniformly parabolic comparison/uniqueness assertion in
Proposition~\ref{prop:regularized-dirichlet}.  Likewise,
$\widehat w|_{[0,T]\times\overline\Omega}$ and $w$ solve the same geometric equation with
$f$ and the same continuous initial values and boundary values on $(0,T]\times\partial\Omega$; their equality follows from
Theorem~\ref{thm:comparison}, equivalently the uniqueness clause of
Proposition~\ref{prop:dirichlet}.  Hence
\[
 \widehat w_n=w_n,\qquad \widehat w=w
 \qquad\text{on }[0,T]\times\overline\Omega.
\]
The semicontinuous limits are therefore formed from the extended sequence.

Define the upper and lower semicontinuous limiting functions on $[0,T^+)\times\overline\Omega$ by
\[
 \overline w(t,x):=\limsup_{\substack{n\to\infty\\(s,y)\to(t,x),\
                    0\le s<T^+,\ y\in\overline\Omega}}\widehat w_n(s,y),
 \qquad
 \underline w(t,x):=\liminf_{\substack{n\to\infty\\(s,y)\to(t,x),\
                    0\le s<T^+,\ y\in\overline\Omega}}\widehat w_n(s,y).
\]
Their boundary values are defined by the same upper and lower limiting procedures from the open space--time domain.  Comparison gives
$-Mt\le \widehat w_n\le1$, so both limiting functions are finite.  In particular,
the extended sequence is uniformly bounded, hence it satisfies the local
uniform-boundedness hypothesis of Lemma~\ref{lem:separate-term-vanishing-derivative-stability} on
every compact subdomain.
By \eqref{eq:divergence-nondivergence}, each $\widehat w_n$ is a viscosity solution of the
nondivergence-form equation
$\partial_t\widehat w_n+\cL^{\eps_n,\nu_n}(x,D\widehat w_n,D^2\widehat w_n)=\widehat f_n$.
Let $\varphi\in C^{1,2}$ be associated with a local maximum or minimum of a semicontinuous
limiting function. If $D_x\varphi\neq0$ at the extremum, Lemma~\ref{lem:away-zero-gradient-convergence}
gives convergence of the principal operators. If $D_x\varphi=0$ and
$D_x^2\varphi=0$ at the extremum, Lemma~\ref{lem:separate-term-vanishing-derivative-stability}
applies on every compact subdomain of $(0,T^+)\times\Omega$, and
\eqref{eq:uniform-small-derivative-bound} gives convergence of the principal term to zero.
Consequently,
$\overline w$ is a subsolution and $\underline w$ is a supersolution of the extended limiting equation in the vanishing-derivative formulation.  By Proposition~\ref{prop:vanishing-derivative-equivalence},
they are, respectively, a relaxed Dirichlet viscosity subsolution and a relaxed Dirichlet viscosity supersolution in the convention of
Definition~\ref{def:viscosity}.

We next identify the parabolic boundary values of the semicontinuous limits. Let
$(t_n,x_n)\to(t_*,y)$, where $t_*\in(0,T^+)$ and $y\in\partial\Omega$.
For all sufficiently large \(n\), either
\(d_\partial(x_n)=0\) or \(0<d_\partial(x_n)<\delta\). If
\(d_\partial(x_n)=0\), then
\[
\widehat w_n(t_n,x_n)=0.
\]
If \(0<d_\partial(x_n)<\delta\), then
Lemma~\ref{lem:regularized-boundary}, applied on \([0,T^+]\), gives
\[
-A_{M,T^+}d_\partial(x_n)
\le \widehat w_n(t_n,x_n)
\le A_{M,T^+}d_\partial(x_n).
\]
Thus this two-sided inequality holds in both cases.
Since $d_\partial(x_n)\to0$,
\[
\limsup_{n\to\infty}\widehat w_n(t_n,x_n)\le0,
\qquad
\liminf_{n\to\infty}\widehat w_n(t_n,x_n)\ge0.
\]
Thus the upper and lower semicontinuous limits satisfy the homogeneous boundary condition on $(0,T^+)\times\partial\Omega$. Moreover, Lemma~\ref{lem:initial-time-control} gives, uniformly in
$n$,
\[
w^0(x)-C_0t
\le \widehat w_n(t,x)
\le w^0(x)+C_0t
\qquad
((t,x)\in[0,T^+)\times\overline\Omega).
\]
If $(t_n,x_n)\to(0,x)$, the continuity of $w^0$ therefore implies
\[
\limsup_{n\to\infty}\widehat w_n(t_n,x_n)\le w^0(x),
\qquad
\liminf_{n\to\infty}\widehat w_n(t_n,x_n)\ge w^0(x).
\]
These inequalities identify the initial values of the semicontinuous limits. By
Proposition~\ref{prop:vanishing-derivative-equivalence}, $\overline w$ and $\underline w$
satisfy the interior subsolution and supersolution inequalities, respectively. If $\zeta\in(0,S)\times\partial\Omega$, the boundary-distance estimate gives
\[
\overline w(\zeta)\le0\le\underline w(\zeta).
\]
If $\zeta=(0,x)$, the initial-time estimate gives
\[
\overline w(0,x)\le w^0(x)\le\underline w(0,x).
\]
Thus, writing $\gamma$ for the prescribed parabolic boundary datum, one has
\[
\overline w(\zeta)-\gamma(\zeta)\le0,
\qquad
\underline w(\zeta)-\gamma(\zeta)\ge0
\qquad(\zeta\in\partial_PQ_S).
\]
These are precisely the first inequalities in the relaxed boundary minimum and maximum
conditions of Definition~\ref{def:viscosity}. Consequently, for every $S\in(0,T^+)$,
$\overline w$ and $\underline w$ are bounded relaxed Dirichlet viscosity sub- and
supersolutions, respectively, in $Q_S$.

Fix $T^\sharp$ with $T<T^\sharp<T^+$.  For every $S$ with
$T<S<T^\sharp$, the established interior and boundary inequalities apply
on $Q_S$.  Its parabolic boundary consists of the sets $\{0\}\times\overline\Omega$ and $(0,S)\times\partial\Omega$, and the
comparison uses the boundary values on those two components.  The semicontinuous boundary bounds already
proved are exactly those required by Theorem~\ref{thm:comparison}.
On $\partial_PQ_S$, the established boundary inequalities give
$\overline w\le\widehat w$ and $\widehat w\le\underline w$, because all three functions
have the common initial value $w^0$ and the common zero boundary value on $(0,S)\times\partial\Omega$.  Applying
Theorem~\ref{thm:comparison} first to the pair
$(\overline w,\widehat w)$ and then to $(\widehat w,\underline w)$ gives
$\overline w\le\widehat w\le\underline w$.  Since always
$\underline w\le\overline w$, comparison therefore identifies
\[
 \overline w=\underline w=\widehat w
 \qquad\text{on }[0,S)\times\overline\Omega.
\]
Since $S\uparrow T^\sharp$ is arbitrary,
\[
 \overline w=\underline w=\widehat w
 \qquad\text{on }[0,T^\sharp)\times\overline\Omega.
\]
Since $T<T^\sharp$, the target set $[0,T]\times\overline\Omega$ is compactly contained in $[0,T^\sharp)\times\overline\Omega$. The equality of the upper and lower semicontinuous limits then implies uniform convergence on the target set.  If uniform convergence failed, there would exist a number $\delta>0$, a
subsequence $n_k\to\infty$, and points
$z_k\in[0,T]\times\overline\Omega$ such that
\[
 \abs{\widehat w_{n_k}(z_k)-\widehat w(z_k)}\ge\delta.
\]
After extracting a further subsequence, $z_k\to z$ by compactness.  If
$\widehat w_{n_k}(z_k)\ge\widehat w(z_k)+\delta$, continuity of $\widehat w$ gives
\[
 \overline w(z)\ge\limsup_{k\to\infty}\widehat w_{n_k}(z_k)
 \ge \widehat w(z)+\delta,
\]
contradicting \(\overline w=\widehat w\). If instead
\(\widehat w_{n_k}(z_k)\le\widehat w(z_k)-\delta\), then continuity of
\(\widehat w\) gives
\[
\underline w(z)
\le\liminf_{k\to\infty}\widehat w_{n_k}(z_k)
\le\widehat w(z)-\delta,
\]
contradicting \(\underline w=\widehat w\).  Thus
$\widehat w_n\to\widehat w$ uniformly on $[0,T]\times\overline\Omega$.
Restricting back to $[0,T]$ proves the asserted convergence of $w_n$ to $w$.
\end{proof}

\section{Proof of the main theorems}\label{sec:main-proofs}
For $v\in X_T$, let $\mathfrak T(v)$ be the restriction to
$[0,T]\times\overline\Omega$ of the unique Dirichlet viscosity solution on
$[0,T^+]\times\overline\Omega$ of
\[
\partial_tw-\cK(x,Dw,D^2w)=f^{v,T}(t,x),\qquad
w|_{\partial\Omega}=0,\qquad w(0)=u^0.
\]
Proposition~\ref{prop:dirichlet} and Lemma~\ref{lem:reaction-inhomogeneity} imply that $\mathfrak T:X_T\to X_T$ is well defined.

\begin{lemma}[Contraction of the reaction map]\label{lem:inhomogeneous-stability-map-contraction}
For every $\lambda>0$ and $v,w\in X_T$,
\begin{equation}\label{eq:reaction-map-contraction}
\norm{\mathfrak T(v)-\mathfrak T(w)}_{\lambda,T}
\le\frac{\mu L_\eta}{\lambda}\norm{v-w}_{\lambda,T}.
\end{equation}
\end{lemma}

\begin{proof}
For \(t\in[0,T]\), Lemma~\ref{lem:inhomogeneous-stability} and
\eqref{eq:inhomogeneity-lipschitz} give
\[
\|\mathfrak T(v)(t)-\mathfrak T(w)(t)\|_{L^\infty(\Omega)}
\le
\mu L_\eta\int_0^t
\|v(s)-w(s)\|_{L^\infty(\Omega)}\,\mathrm ds.
\]
Substitution of
\[
\|v(s)-w(s)\|_{L^\infty(\Omega)}
\le e^{\lambda s}\|v-w\|_{\lambda,T}
\]
into the preceding integral inequality gives
\[
e^{-\lambda t}
\|\mathfrak T(v)(t)-\mathfrak T(w)(t)\|_{L^\infty(\Omega)}
\le
\frac{\mu L_\eta}{\lambda}
(1-e^{-\lambda t})\|v-w\|_{\lambda,T}.
\]
Taking the supremum over \(t\in[0,T]\) proves
\eqref{eq:reaction-map-contraction}.
\end{proof}

\begin{proof}[Proof of Theorem~\ref{thm:wellposed}]
Fix \(T>0\) and choose \(\lambda>\mu L_\eta\). By
Lemma~\ref{lem:inhomogeneous-stability-map-contraction}, the map
\(\mathfrak T:X_T\to X_T\) is a contraction. The Banach fixed-point theorem
therefore yields a unique \(u^{(T)}\in X_T\) satisfying
\[
\mathfrak T(u^{(T)})=u^{(T)}.
\]
By the definition of \(\mathfrak T\),
\[
f^{u^{(T)},T}(t,x)
=-\mu R_\eta(x,u^{(T)}(t,x))
\qquad(0\le t\le T),
\]
so \(u^{(T)}\) solves \eqref{eq:limit-flow}. Conversely, every continuous
Dirichlet viscosity solution of \eqref{eq:limit-flow} is a fixed point of
\(\mathfrak T\), and is therefore equal to \(u^{(T)}\).

If \(0<T_1<T_2\), then the restriction of \(u^{(T_2)}\) to
\([0,T_1]\times\overline\Omega\) is a fixed point of the map defined on
\(X_{T_1}\). Uniqueness gives
\[
u^{(T_2)}=u^{(T_1)}
\qquad\text{on }[0,T_1]\times\overline\Omega.
\]
The compatible family \(\{u^{(T)}\}_{T>0}\) defines a unique global
solution.

Since \(R_\eta(x,0)=0\), the constant function \(0\) is a subsolution. Since
\(R_\eta(x,1)\ge0\), the constant function \(1\) is a supersolution. Their
initial values and values on $(0,\infty)\times\partial\Omega$ are ordered with those of \(u\); hence Theorem~\ref{thm:reaction-comparison}
gives
\[
0\le u\le1.
\]

Let \(u\) and \(v\) correspond to initial data \(u^0\) and \(v^0\), and
set
\[
c=\|u^0-v^0\|_{L^\infty(\Omega)}.
\]
Since $c$ is independent of $t$ and $x$,
\[
\partial_t(v+c)=\partial_t v,\qquad D_x(v+c)=D_xv,\qquad
D_x^2(v+c)=D_x^2v.
\]
Let $\varphi\in C^{1,2}$ and suppose that $v+c-\varphi$ has a local
minimum at $(t_0,x_0)$. Then $v-(\varphi-c)$ has a local minimum at the
same point. The viscosity supersolution inequality for $v$ and the
monotonicity of $R_\eta(x,\cdot)$ yield
\begin{align*}
&\partial_t\varphi(t_0,x_0)
+\cLgeoUpper\!\left(x_0,D_x\varphi(t_0,x_0),D_x^2\varphi(t_0,x_0)\right)
+\mu R_\eta\bigl(x_0,\varphi(t_0,x_0)\bigr)\\
&\quad\ge
\mu\left[
R_\eta\bigl(x_0,\varphi(t_0,x_0)\bigr)
-R_\eta\bigl(x_0,\varphi(t_0,x_0)-c\bigr)
\right]\ge0.
\end{align*}
Therefore $v+c$ is a viscosity supersolution of the equation satisfied by
$u$. Moreover,
\[
u^0\le v^0+c,
\qquad
0\le c\quad\text{on $(0,\infty)\times\partial\Omega$}.
\]
Theorem~\ref{thm:reaction-comparison} yields \(u\le v+c\). Interchanging \(u\) and \(v\) gives
\(v\le u+c\), and therefore
\[
\|u(t)-v(t)\|_{L^\infty(\Omega)}
\le
\|u^0-v^0\|_{L^\infty(\Omega)}
\qquad(t\ge0).
\]
\end{proof}

For $v\in X_T$ and $\eps,\nu>0$, let
$\mathfrak T^{\eps,\nu}(v)$ be the restriction to $[0,T]\times\overline\Omega$ of the unique
Dirichlet viscosity solution on $[0,T^+]\times\overline\Omega$ of
\begin{equation}\label{eq:regularized-map-equation}
\begin{cases}
\partial_tw-\nu\Delta w
-A_\eps(\nabla w)\operatorname{div}\!\left(G\dfrac{\nabla w}{A_\eps(\nabla w)}\right)
=f^{v,T}(t,x),&(t,x)\in(0,T^+)\times\Omega,\\
w=0,&(t,x)\in(0,T^+)\times\partial\Omega,\\
w(0,x)=u^0(x),&x\in\overline\Omega.
\end{cases}
\end{equation}
Proposition~\ref{prop:regularized-dirichlet} makes the map well defined entirely within the
Dirichlet viscosity framework.

\begin{lemma}[Uniform contraction of the regularized maps]\label{lem:regularized-map-contraction}
For every $\lambda>0$,
\[
\norm{\mathfrak T^{\eps,\nu}(v)-\mathfrak T^{\eps,\nu}(w)}_{\lambda,T}
\le\frac{\mu L_\eta}{\lambda}\norm{v-w}_{\lambda,T},
\]
with a contraction constant independent of $\eps$ and $\nu$.
\end{lemma}

\begin{proof}
Lemma~\ref{lem:regularized-inhomogeneous-stability} gives, for every
\(t\in[0,T]\),
\[
\|\mathfrak T^{\eps,\nu}(v)(t)-
\mathfrak T^{\eps,\nu}(w)(t)\|_{L^\infty(\Omega)}
\le
\mu L_\eta\int_0^t
\|v(s)-w(s)\|_{L^\infty(\Omega)}\,\mathrm ds.
\]
The weighted-norm calculation in the proof of
Lemma~\ref{lem:inhomogeneous-stability-map-contraction} then yields
\[
\|\mathfrak T^{\eps,\nu}(v)-
\mathfrak T^{\eps,\nu}(w)\|_{\lambda,T}
\le
\frac{\mu L_\eta}{\lambda}
\|v-w\|_{\lambda,T}.
\]
The constant is independent of \(\eps\) and \(\nu\).
\end{proof}

\begin{proposition}[Well-posedness of the regularized problem]\label{prop:regularized-full}
For every $\eps,\nu>0$, problem \eqref{eq:regularized-flow} has a unique global
Dirichlet viscosity solution $u^{\eps,\nu}$ that attains the prescribed initial and
boundary data. It satisfies $0\le u^{\eps,\nu}\le1$.
\end{proposition}

\begin{proof}
Fix \(T>0\) and choose \(\lambda>\mu L_\eta\). By
Lemma~\ref{lem:regularized-map-contraction},
\(\mathfrak T^{\eps,\nu}:X_T\to X_T\) is a contraction. The Banach
fixed-point theorem gives a unique fixed point \(u^{\eps,\nu,(T)}\), and the
fixed-point identity is equivalent to \eqref{eq:regularized-flow} on
\([0,T]\). If \(T_1<T_2\), uniqueness on \([0,T_1]\) identifies
\(u^{\eps,\nu,(T_2)}|_{[0,T_1]}\) with
\(u^{\eps,\nu,(T_1)}\). Hence the finite-time solutions define a unique
global Dirichlet viscosity solution.

The constant functions \(0\) and \(1\) are, respectively, a subsolution and
a supersolution. Proposition~\ref{prop:regularized-reaction-comparison} gives
\[
0\le u^{\eps,\nu}\le1.
\]
\end{proof}

\begin{proof}[Proof of Theorem~\ref{thm:limit}]
Fix \(T>0\) and choose \(\lambda>\mu L_\eta\). Let \(u\) be the fixed point
of \(\mathfrak T\), and let
\(u_n=u^{\eps_n,\nu_n}\) be the fixed point of
\(\mathfrak T^{\eps_n,\nu_n}\). The functions
\(\mathfrak T^{\eps_n,\nu_n}(u)\) and \(\mathfrak T(u)\) have the same
initial data and boundary data on $(0,T]\times\partial\Omega$ and satisfy equations with the common continuous
inhomogeneity
\[
f^{u,T}(t,x)
=-\mu R_\eta\bigl(x,(\mathcal E_Tu)(t,x)\bigr).
\]
Thus the reaction is evaluated at the limiting fixed point \(u\) in both
equations to which Lemma~\ref{lem:vanishing-regularization-selection} is
applied. At this stage, no limit is taken directly in
\(R_\eta(x,u^{\eps_n,\nu_n})\). The convergence for this common
inhomogeneity is subsequently transferred to the nonlinear fixed points by
the contraction estimate below. Lemma~\ref{lem:vanishing-regularization-selection} gives
\[
\delta_n(T)
:=
\|\mathfrak T^{\eps_n,\nu_n}(u)-\mathfrak T(u)\|_{\lambda,T}
\longrightarrow0.
\]
Using the fixed-point identities and
Lemma~\ref{lem:regularized-map-contraction},
\begin{align*}
\|u_n-u\|_{\lambda,T}
&\le
\|\mathfrak T^{\eps_n,\nu_n}(u_n)
-\mathfrak T^{\eps_n,\nu_n}(u)\|_{\lambda,T}
+\delta_n(T)\\
&\le
\frac{\mu L_\eta}{\lambda}\|u_n-u\|_{\lambda,T}
+\delta_n(T).
\end{align*}
Therefore
\[
\|u_n-u\|_{\lambda,T}
\le
\frac{\delta_n(T)}{1-\mu L_\eta/\lambda}
\longrightarrow0.
\]
Since
\[
e^{-\lambda T}\|z\|_{C([0,T]\times\overline\Omega)}
\le
\|z\|_{\lambda,T}
\le
\|z\|_{C([0,T]\times\overline\Omega)},
\]
convergence in the weighted norm is equivalent to uniform convergence on
\([0,T]\times\overline\Omega\). This proves \eqref{eq:localuniform} for
every sequence \((\eps_n,\nu_n)\to(0,0)\).
\end{proof}

\section{Conclusion}\label{sec:conclusion}

The limiting weighted geometric equation with monotone reaction admits a unique
global viscosity solution, preserves the interval $[0,1]$, and is nonexpansive in the
supremum norm. Under the mean-convexity and boundary-neighborhood assumptions, this
solution extends continuously to the boundary and attains the prescribed homogeneous
Dirichlet data.

The zero-gradient analysis distinguishes the direct joint limits of the regularized
operators from the geometric directional extensions and identifies the latter through
the vanishing-derivative formulation. Treating the principal and reaction terms
separately then yields uniform convergence of the regularized Dirichlet viscosity
solutions on every finite time interval for arbitrary joint decay
$(\eps,\nu)\to(0,0)$. Together with the fixed-parameter classical theory in
\cite{Uba2026Regularized}, these results provide a rigorous analytical basis for the
regularized subjective-surface model in 3D and 3D+time microscopy segmentation.

\end{document}